\documentclass[11pt,a4paper]{article} 
\usepackage[english]{babel}
\usepackage[utf8]{inputenc}
\usepackage[T1]{fontenc}
\usepackage{color}
\usepackage{graphicx}
\usepackage{tikz}
\usepackage{tkz-tab}
\usetikzlibrary{calc}
\usetikzlibrary {arrows.meta}
\usepackage[justification=centering]{caption}
\usepackage{wrapfig}
\usepackage{amsmath}
\usepackage{amssymb}
\usepackage{pdfpages}
\usepackage{xcolor}
\usepackage{cancel}
\usepackage{multicol}
\usepackage[export]{adjustbox}
\usepackage{caption}
\usepackage{subcaption}
\usepackage{comment}
\usepackage{amsthm}
\usepackage{calc}
\usepackage{subfiles}
\usepackage{epstopdf}
\usepackage{bm}
\usepackage{setspace}
\usepackage{hyperref}

\newtheorem{rems*}{Remarks}
\newtheorem{rem*}{Remark}

\newtheorem{lem}{Lemma}[section]
\newtheorem{prop}{Proposition}[section]
\newtheorem{thm}{Theorem}[section]
\newtheorem{rem}{Remark}[section]
\newtheorem{cor}{Corollary}[section]
\makeatletter

\@addtoreset{equation}{section}
\makeatother

\newcommand{\myol}[2][3]{{}\mkern#1mu\overline{\mkern-#1mu#2}}
\newcommand{\e}[1]{\mathrm{e}^{#1}}

\newcommand{\vecapp}[1]{\underline{#1}}
\newcommand{\vecdiff}[1]{\mathrm{\textbf{e}}_{#1}}
\newcommand{\vecbound}[1]{\epsilon_{#1}}
\newcommand{\vectrunc}[2][3]{#1\left|\right._{#2}}
\newcommand{\convomat}[1]{\mathrm{M}(#1)}
\newcommand{\normenu}[1]{\left\|#1\right\|_\nu}

\newcommand{\R}{\mathbb R}
\newcommand{\bu}{\bar u}
\newcommand{\bv}{\bar v}

\usepackage{geometry}
\title{Computer-assisted proofs for the many steady states of a chemotaxis model with local sensing}

\author{Maxime Breden \thanks{CMAP, CNRS, \'Ecole polytechnique, Institut Polytechnique de
Paris, 91120 Palaiseau, France. \texttt{maxime.breden@polytechnique.edu}} $\quad$ Maxime Payan \thanks{CMAP, CNRS, \'Ecole polytechnique, Institut Polytechnique de
Paris, 91120 Palaiseau, France. \texttt{maxime.payan@polytechnique.edu}} } 
\date{\today}

\begin{document}
\setlength{\parindent}{0cm}
\setlength{\abovedisplayskip}{5pt}
\setlength{\abovedisplayshortskip}{5pt}
\setlength{\belowdisplayskip}{5pt}
\setlength{\belowdisplayshortskip}{5pt}
\renewcommand{\arraystretch}{0.7} \setlength\arraycolsep{1mm}

\maketitle

\begin{abstract}
    We study the steady states of a system of cross-diffusion equations arising from the modeling of chemotaxis with local sensing, where the motility is a decreasing function of the concentration of the chemical. In order to capture the many different equilibria that sometimes co-exist, we use computer-assisted proofs: Given an approximate solution obtained numerically, we apply a fixed-point argument in a small neighborhood of this approximate solution to prove the existence of an exact solution nearby. This allows us to rigorously study the steady states of this cross-diffusion system much more extensively than what previously possible with purely pen-and-paper techniques. Our computer-assisted argument makes use of Fourier series decomposition, which is common in the literature, but usually restricted to systems with polynomial nonlinearities. This is not the case for the model considered in this paper, and we develop a new way of dealing with some nonpolynomial nonlinearities in the context of computer-assisted proofs with Fourier series. 
\end{abstract}

\section{Introduction} \label{sec:intro}

Chemotaxis is one of the critical mechanisms through which the behavior of bacteria and other small organisms can be understood, and the patterns they form explained. Following the seminal papers of Patlak~\cite{Pat53}, Keller and Segel~\cite{KelSeg70,KelSeg71}, chemotaxis models have been intensively studied in the mathematical literature. However, even the simplest patterns generated by such models, i.e. (nonhomogenenous) stationary solutions, often remains very challenging to rigorously analyze.

The purpose of this paper is to provide a detailed study of the (positive) steady states of the following chemotaxis model:
\begin{align}
\label{eq:syst_para}
\left\{
\begin{aligned}
&\partial_t u = \Delta\left(\gamma(v)u\right) + \sigma u(1-u) \qquad & \text{in } (0,\infty)\times\Omega, \\
&\partial_t v = d\Delta v + u - v & \text{in } (0,\infty)\times\Omega, \\
&\partial_n \left(\gamma(v)u\right) = 0 = \partial_n v  & \text{on } (0,\infty)\times\partial\Omega. \\
\end{aligned}
\right.
\end{align}
Here, $u=u(t,x)$ denotes the bacteria or cell density, $v=v(t,x)$ the density of a chemical produced by the bacteria, and $\Omega$ is a bounded domain of $\R^N$. Throughout the paper, we assume that $\gamma$ is smooth and positive on $(0,\infty)$, $d>0$ and $\sigma\geq 0$. The most interesting situation, at least if one wants to maybe observe pattern formation, is the one where the bacteria are in some sense attracted by the chemical. In~\eqref{eq:syst_para}, this can be modeled by taking $\gamma$ decreasing. We note that, while we will only consider examples where $\gamma$ is indeed decreasing, this assumption is in fact not needed for our results (and neither is the positivity of $\sigma$).

System~\eqref{eq:syst_para} is in fact a specific case of one of the original models introduced by Keller and Segel in~\cite{KelSeg71} (up to the logistic reaction term), where the equation on $u$ writes
\begin{align}
\label{eq:minimalKS}
    \partial_t u = \nabla\cdot\left(\mu(v)\nabla u - \chi(v)u\nabla v\right).
\end{align}
The diffusive term in~\eqref{eq:syst_para} is then obtained by assuming that the diffusivity $\mu$ and the chemosensitivity $\chi$ are related by $\mu' = -\chi$, which corresponds to assuming that the bacteria are only sensitive to the concentration of $v$, but not to its gradient (\emph{local sensing}). For more details on these modeling interpretations, we refer to~\cite{DKTY19} and the references therein. System~\eqref{eq:syst_para} was also reintroduced in~\cite{Fu++12}, as an attempt to explain the stripe patterns observed in~\cite{Liu++11}, and it has become the subject of recent scrutiny from the mathematical community. 

For $\sigma>0$, global well-posedness for the initial value problem associated to~\eqref{eq:syst_para} was first established in~\cite{JinKimWan18} in dimension at most 2, and under some further assumptions on $\gamma$, encompassing prototypical examples such as
\begin{align}
\label{eq:examples_gamma}
    \gamma(v) = \frac{a}{(1+bv)^c}, \quad \gamma(v) = \frac{a}{(1+e^{bv})^c}\quad\text{or}\quad \gamma(v) = 1-\frac{v}{\sqrt{1+v^2}},
\end{align}
for $a,b,c> 0$. Global well-posedness was then also obtained in higher dimension in~\cite{LiuXu19,WanWan19}, under an additional assumption on $\sigma$, which has to be sufficiently large.

The case $\sigma=0$ is more delicate, as the decay of $\gamma$ at infinity seems to play a critical role. Indeed, for simple examples of $\gamma$, global existence could be established when $\gamma$ decays algebraically, whereas a critical mass threshold leading to blow-ups was identified for exponentially decaying $\gamma$'s. More precisely, assuming $\gamma(v) = \frac{a}{v^c}$ with $a$ small enough, the global existence of bounded solutions was obtained in~\cite{YooKim17} in any dimension, whereas global weak solutions are obtained in~\cite{DKTY19} for $\gamma(v) = \frac{a}{1+v^c}$, without any smallness assumption on $a$ but with restrictions on $c$ depending on the dimension. On the other hand, when $\gamma(v) = e^{-bv}$, blow-up can occur if the initial mass is large enough~\cite{JinWan20}, although only in infinite time~\cite{BurLauTre21,FujJia21}, which contrasts with the famous behavior for the minimal Keller-Segel model (i.e. when taking $\mu$ and $\chi$ constant in~\eqref{eq:minimalKS}), where blow-up occurs in finite time~\cite{HerVel97}.

Regarding steady states of system~\eqref{eq:syst_para}, the first trivial observation is that $(0,0)$ and $(1,1)$ are the only homogenenous steady states when $\sigma>0$, whereas $(c,c)$ is a steady state for any $c\in\R$ when $\sigma = 0$. Assuming $\sigma>0$, a straightforward linear stability analysis of the positive steady state $(1,1)$ yields the following necessary conditions for its linear instability (which are also sufficient, up to choosing a compatible domain):
\begin{align}
\label{eq:lin_stab}
\left\{
\begin{aligned}
    \sigma d  + \gamma(1) + \gamma'(1) &< 0, \\
    (\sigma d  + \gamma(1) + \gamma'(1))^2 - 4\sigma d \gamma(1) &> 0.
\end{aligned}
\right.
\end{align}
In particular, one immediately sees that, when the product $\sigma d$ is large enough, the trivial steady state $(1,1)$ is linearly stable. Using a priori bounds and the Leray-Schauder index, one can go further and show that whenever $\sigma$ is large enough (with a threshold depending on $d$), or whenever $d$ is large enough (with a threshold depening on $\sigma$), there exists no other positive steady state of~\eqref{eq:syst_para}, see~\cite{MPW19}. 

Regardless of the (nonnegative) value of $\sigma d$, one also sees that, if $\gamma$ is such that $\gamma(1) + \gamma'(1) \geq 0$, then $(1,1)$ must be linearly stable. On the other hand, if $\gamma(1) + \gamma'(1) < 0$, condition~\eqref{eq:lin_stab} does hold for at least some values of $\sigma$ and $d$. One can then use local bifurcation theory to predict the apparition of branches of nonhomogeneous steady states, and further usage of the Leray-Schauder index and of other global bifurcation tools can then provide some further insight on the existence of nontrivial steady states. Such bifurcation analyses were carried out in~\cite{MPW19,WX21}, using $\sigma$ as a bifurcation parameter. These works do provide existence results regarding non trivial steady states, with typical statements of the form: ``if $\sigma$ belongs to a given interval (depending on the other parameters), there exists at least one non trivial steady state of~\eqref{eq:syst_para}''. However, such results are still very far from providing a full description of the steady states of~\eqref{eq:syst_para}, especially when contrasted with numerical experiments, which suggest the existence of intricate bifurcation diagrams of steady states, featuring many co-existing steady states for some given value of $\sigma$ (see Figure~\ref{fig:continuous-diag}).

\begin{figure}[ht]
    \centering
    \includegraphics[width=0.8\paperwidth]{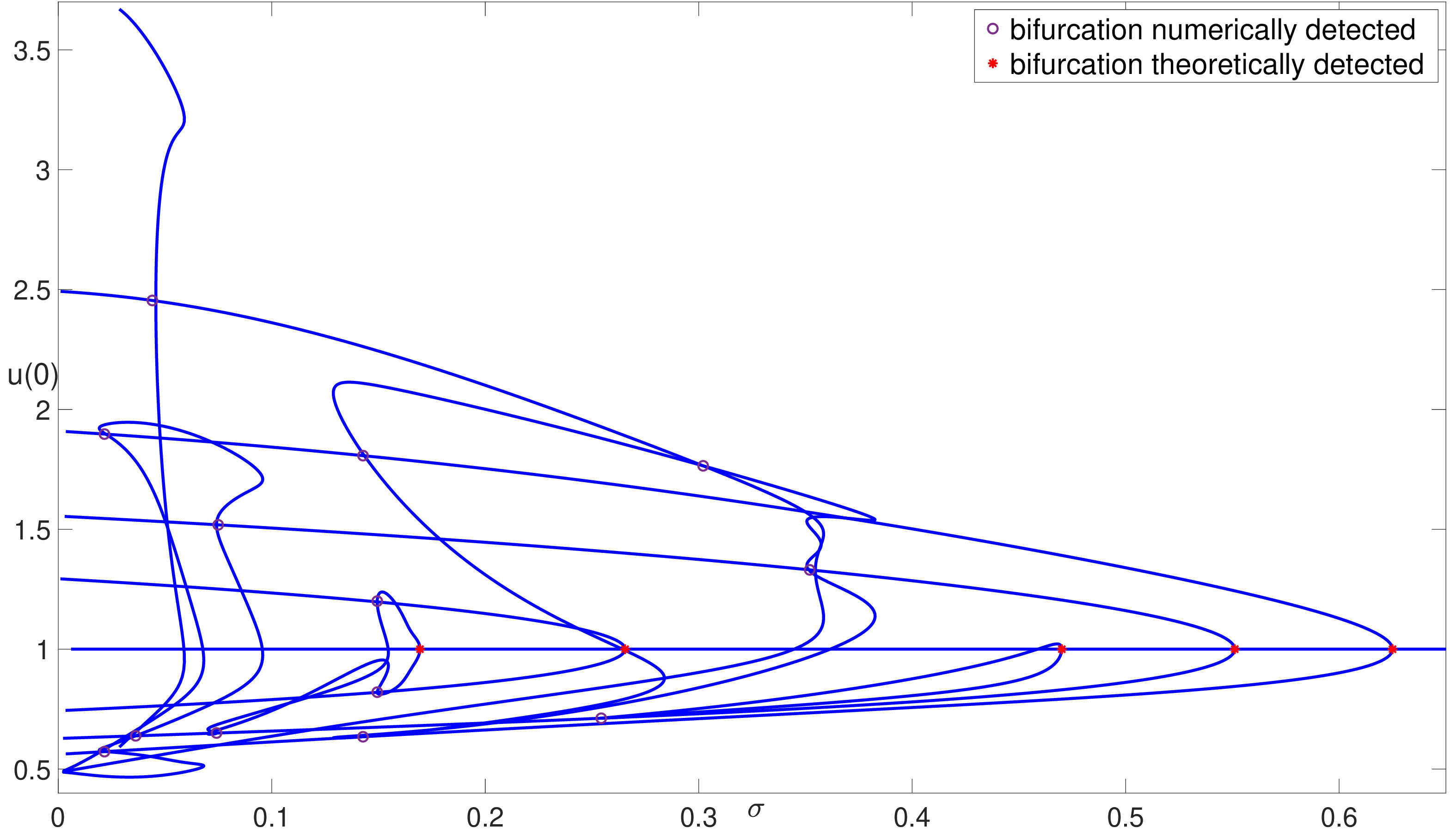}
    \caption{Numerical bifurcation diagram of steady states of the system \eqref{eq:syst_para}, with $\Omega=(0,3\pi)$, $d=1$ and $\gamma(x) = \dfrac{1}{1+x^9}$. }
    \label{fig:continuous-diag}
\end{figure}

When $\sigma = 0$, the linear instability condition for a given steady state $(c,c)$ reduces to $\gamma(c) + \gamma'(c)<0$ (see~\cite{DKTY19} for the details).
 Here again, there are already results going beyond this linear analysis. First negative ones, assuming that $\gamma$ decay sublinearly, or more precisely, that $v\mapsto v\gamma(v)$ is increasing. Under this hypothesis, it is shown in~\cite{DesLauTreWin23} that the system~\eqref{eq:syst_para} admits a Lyapunov functional, which prevents the existence of any nontrivial steady state. Second, positive results, when $\gamma(c) + \gamma'(c)<0$, which rely again on global bifurcation theory~\cite{WX21}. However, in the $\sigma=0$ case also, the available results fail to encompass the variety of steady states that can be observed numerically.

 In this paper, we develop a computer-assisted methodology providing a rigorous description of the steady states of~\eqref{eq:syst_para}, which is essentially as complete and as precise as what can be observed numerically. More precisely, for any candidate approximate steady state denoted $(\bu,\bv)$ of~\eqref{eq:syst_para}, we provide sufficient conditions ensuring the existence of a true steady state in a small and explicit neighborhood of $(\bu,\bv)$. Given any $(\bu,\bv)$, typically obtained numerically, we can then explicitly check whether these conditions hold, and if they do conclude to the existence of a nearby true steady state of~\eqref{eq:syst_para}.

 \begin{rem}
    In this paper, we only focus on the one dimensional case, which is already very rich, and write $\Omega=(a,b)$, but higher spatial dimensions could also be investigated (in a very similar way for rectangular domains, but with more challenging adaptations for more complicated geometries).
\end{rem}

As an example, we show below the kind of output that can be expected from this procedure, for two different choices of $\gamma$.
\begin{figure}[ht]
    \centering
    \includegraphics[width=0.6\paperwidth]{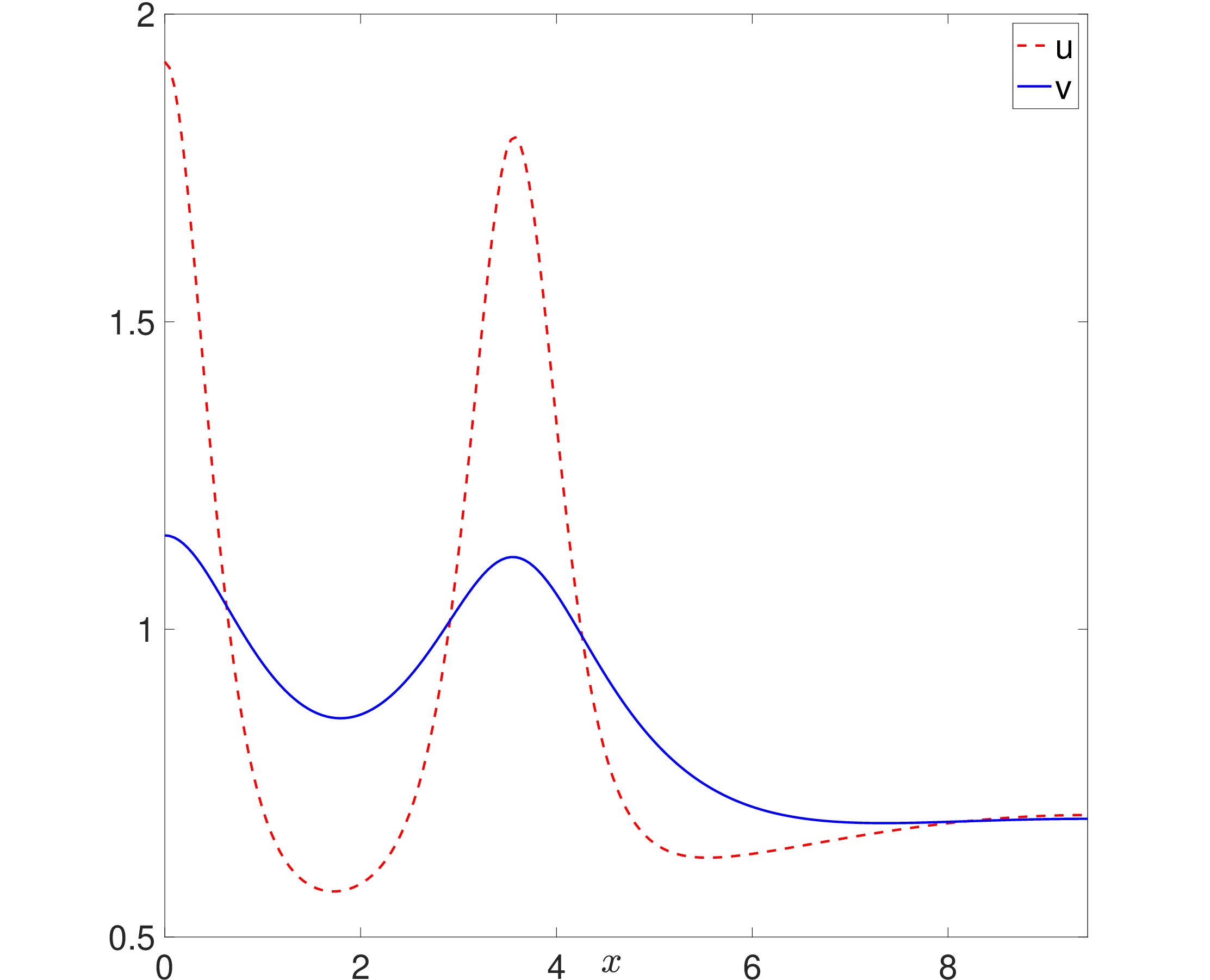}
    \caption{An approximate steady state of \eqref{eq:syst_para}, associated to Theorem~\ref{th:single_sol}.}
    \label{fig:single}
\end{figure}

 \begin{thm}
 \label{th:single_sol}
     Let $\bar{u},\bar{v}$ be the functions depicted on Figure~\ref{fig:single} and whose precise description in terms of Fourier coefficients can be found at \cite{GitHub}. There exists a smooth steady states $(u,v)$ of the system~\ref{eq:syst_para} with the parameters $\sigma = 0.053$, $d = 1$, $\Omega=(0,3\pi)$ and $\gamma(x) = \dfrac{1}{1+x^9}$, such that
     \begin{align*}
         \sup_{x\in\Omega} |u(x)-\bar{u}(x)| + \sup_{x\in\Omega} |v(x)-\bar{v}(x)| \leq 2.5197\times 10^{-8}.
     \end{align*}
 \end{thm}

Notice that we not only get an existence statement regarding steady states of~\eqref{eq:syst_para}, but also very quantitative information regarding the obtained steady state, as we have an explicit approximation together with small error bounds.

\begin{rem}
    The error bound in Theorem~\ref{th:single_sol} is stated in the $C^0$ norm for the sake of simplicity, but is actually first obtained in a stronger norm, which will be introduced in Section~\ref{sec:NK}.
\end{rem}

\begin{figure}[ht]
    \centering
    \includegraphics[width=0.6\paperwidth]{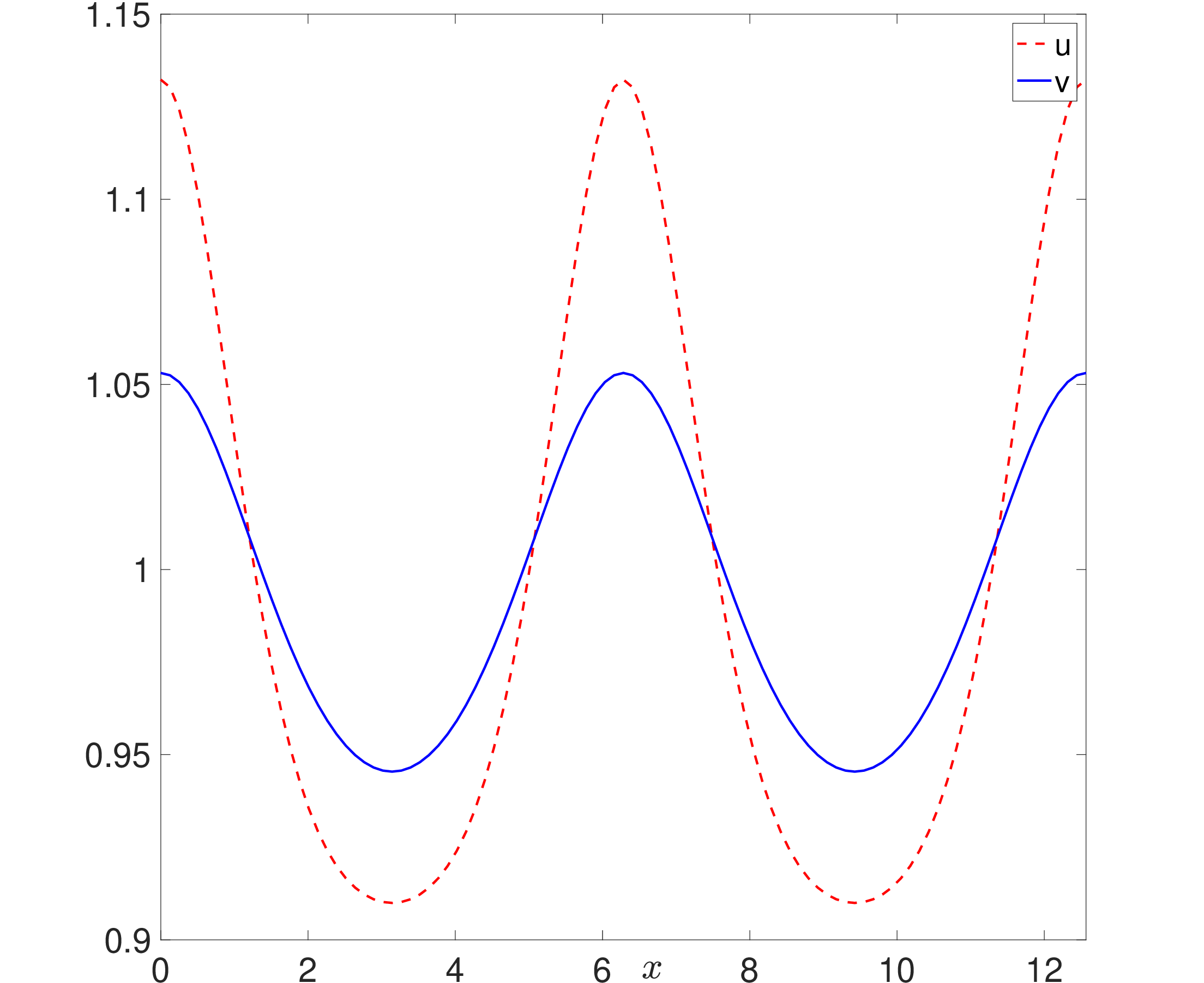}
    \caption{An approximate steady state of \eqref{eq:syst_para}, associated to Theorem~\ref{th:WX}.}
    \label{fig:thWX}
\end{figure}

Here is another example, for a different choice of motility function $\gamma$.
\begin{thm} \label{th:WX}
    Let $\bar{u},\bar{v}$ be the functions depicted on Figure~\ref{fig:thWX} and whose precise description in terms of Fourier coefficients can be found at \cite{GitHub}. There exists a smooth steady states $(u,v)$ of the system~\ref{eq:syst_para} with the parameters $\sigma = 0.6$, $d = 1$, $\Omega=(0,4\pi)$ and $\gamma(x) = \dfrac{1}{1+\e{9(x-1)}}$, such that
     \begin{align*}
         \sup_{x\in\Omega} |u(x)-\bar{u}(x)| + \sup_{x\in\Omega} |v(x)-\bar{v}(x)| \leq 1.6956\times 10^{-12}.
     \end{align*}
\end{thm}

\begin{rem}
    Theorem~\ref{th:WX} corroborates the results of~\cite[Fig.3]{WX21}, where an approximate steady state resembling the one of our Figure~\ref{fig:thWX} was obtained numerically. The techniques presented in this paper allow to rigorously establish the existence of an exact solution near this approximation.
\end{rem}

We can of course repeat such an argument for different approximate solutions, and for different choices of parameters. Putting all these results together, we indeed obtain a rigorous description of the full picture, or at least of the full picture that was found numerically: we can essentially validate any solution that we found numerically, but cannot exclude the existence of other solutions that may have escaped our numerical investigation.  

\begin{figure}[ht]
    \centering
    \includegraphics[width=0.8\paperwidth]{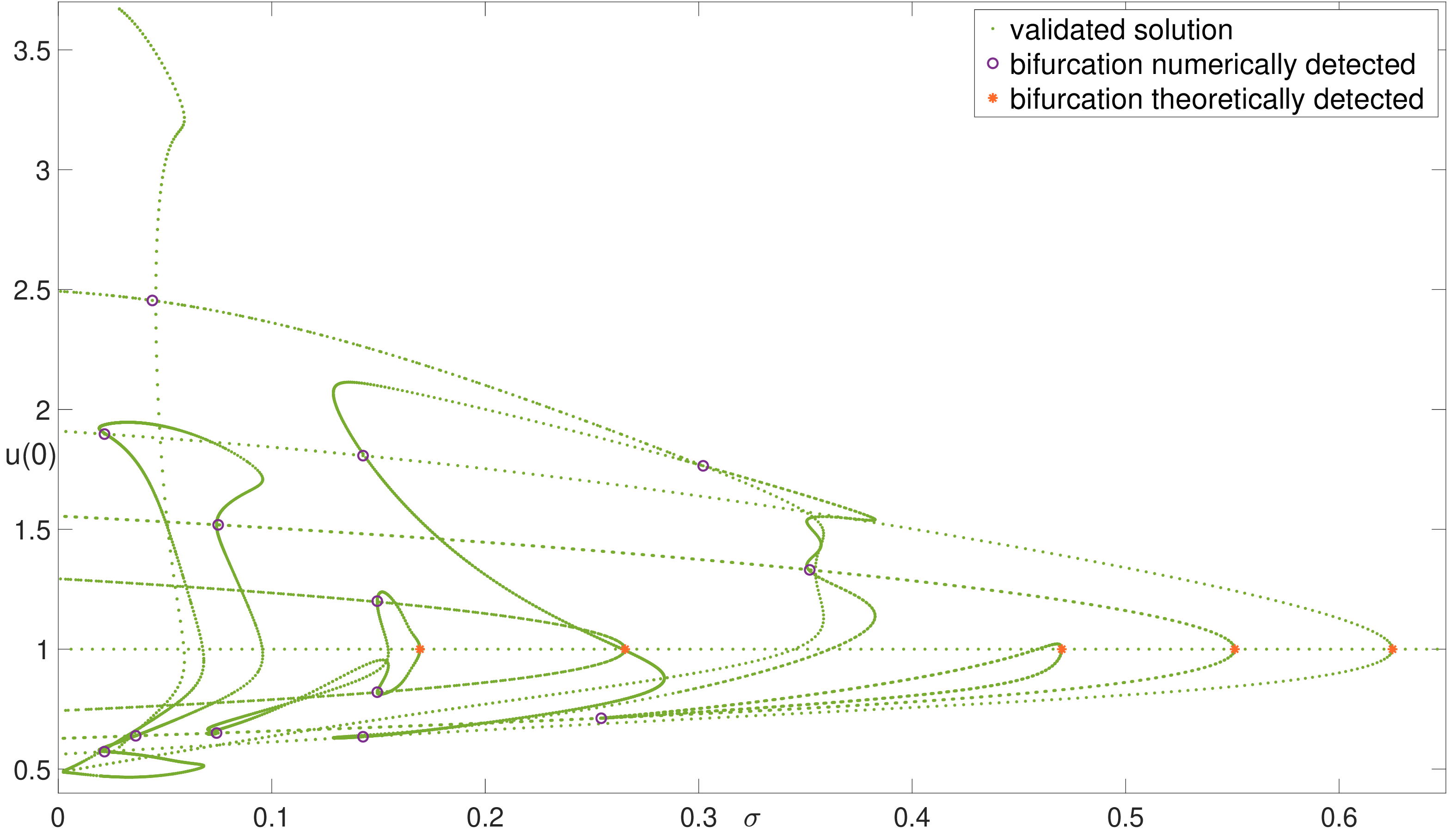}
    \caption{Validated steady states of the system \eqref{eq:syst_para}, with $\Omega=(0,3\pi)$, $d=1$ and $\gamma(x) = \dfrac{1}{1+x^9}$, associated to Theorem~\ref{th:bif_diag}}
    \label{fig:th-discrete-diag}
\end{figure}

 \begin{thm}
 \label{th:bif_diag}
    Each point of the Figure~\ref{fig:th-discrete-diag} represents a steady state for \eqref{eq:syst_para}, with $\Omega=(0,3\pi)$, $d=1$ and $\gamma(x) = \dfrac{1}{1+x^9}$, for the $\sigma$ plotted in x-coordinates.
 \end{thm}

\begin{rem}
    While it is beyond the scope of this work, let us mention that one may try to go one step beyond Theorem~\ref{th:bif_diag}, and really validate the continuous bifurcation diagram depicted in Figure~\ref{fig:continuous-diag}, instead of only the discreteized version of Figure~\ref{fig:th-discrete-diag}. This would require combining the point-wise (in the parameters) theorems obtained here with two extra ingredients. First, rigorous continuation techniques (see e.g.~\cite{Bre23,BreLesVan13,BerLesMis10,Wan18}), in order to validate entire branches of steady states. Second, a desingularization procedure allowing to handle the neighborhood of the pitchfork bifurcation points (see e.g.~\cite{AriGazKoc21,LesSanWan17,RizSanWan23}).
\end{rem}

Computer-assisted techniques for studying elliptic PDEs go back to the pioneering works of Nakao~\cite{Nak88} and Plum~\cite{Plu92}, and have been more and more widely used since then~\cite{AriKoc10,DayLesMis07,Gom19,NakPluWat19,Ois95,TakLiuOis13,BerWil17,Yam98,ZglMis01}. However, their applicability was limited to semilinear problems until very recently, and we will deal with the quasilinear structure of~\eqref{eq:syst_para} by using the techniques introduced in~\cite{B22}. 

Several of our estimates are based on Fourier series decomposition of functions, which has proven to be a very efficient framework for conducting computer-assisted proofs. However, this means that~\eqref{eq:syst_para} present an additional difficulty, namely the fact that the typically $\gamma$'s we are interested in cannot be polynomial (from a modeling point of view, reasonable choices of $\gamma$ should go to $0$ at $+\infty$). In the context of computer-assisted proofs using Fourier series, nonpolynomial terms are usually dealt with with tools based on automatic differentiation or polynomial embeddings~\cite{Hen21,LesMirRan16}, which have proven successful but still suffer from several shortcomings. In particular, they are usually limited to one-dimensional problems, they require enlarging the system with additional equations and unknowns, and for elliptic equations like the steady state problem associated to~\eqref{eq:syst_para}, they destroy the second order structure of the system, which in turn makes the computer-assisted proof harder. In this work, we develop an alternate approach to deal with several types of nonlinear terms, including rational functions, exponential functions, and compositions thereof, which does not necessitate the introduction of new variables and preserves the second order structure of the system.

The remainder of the paper is organized as follows. In Section~\ref{sec:NK}, we introduce some notations and spaces in which we will study the problem, together with a version of the Newton-Kantorovich theorem that will allow us to prove the existence of an exact solution near a numerical approximation. In Section~\ref{sec:tool}, we explain how to deal with the nonpolynomial nonlinearities $\gamma$ considered in this work. The bounds required to apply the Newton-Kantorovich theorem are then obtained in Section~\ref{sec:bounds}. Finally, Section~\ref{sec:proofs} contains some details regarding the proofs of Theorem~\ref{th:single_sol}, Theorem~\ref{th:WX} and Theorem~\ref{th:bif_diag}. All the computations presented in this paper have been implemented in \textsc{Matlab}, using the \textsc{Intlab} package~\cite{intlabRump} for interval arithmetic. The computer-assisted parts of the proofs can be reproduced using the code available at~\cite{GitHub}.

\section{Zero-finding problem and the Newton-Kantorovitch Theorem} \label{sec:NK}
In this section, we introduce the first useful notations to start building a bridge between the numerical aspects (to get numerical solutions) and the theoretical ones (to obtain theorems). To that end, we rewrite our problem as a zero-finding problem, for a single function $F$ defined on a well chosen Banach space. Then, we 
introduce the Newton-Kantorovitch, Theorem~\ref{th:NK}, and its hypotheses. This theorem is the tool which we rely on to show the existence of solutions. 

\subsection{Theoretical framework and associated notations}
Let $u$ be a smooth function defined on $\Omega = (a,b) \subset \mathbb{R}$, $a < b$. We will describe all the objects useful for our demonstration. Since we are in dimension one with homogeneous Neumann conditions at the boundary, we look for solutions as Fourier cosine series, i.e.
\begin{equation*}
     u(x) = \textbf{u}_0 + 2\sum_{n=1}^{+\infty} \textbf{u}_n \cos(\dfrac{n\pi}{b-a}\, (x-a))  .
\end{equation*}
For any smooth function $u$, we denote by $\textbf{u}$ the sequence of its Fourier coefficients $(\textbf{u}_n)_{n\in\mathbb{N}}$. In particular, it will help to use $\textbf{0} = (0,0, \dots )$ and $\textbf{1} = (1, 0, \dots)$. We denote by $\mathcal{F}$ the Fourier transform, i.e. the map such that $\mathcal{F}(u) = \textbf{u}$. 

In order to estimate our objects in the right spaces, we have to choose a norm. We take a weighted $\ell^1$ norm with exponential decay on our sequences, i.e., for any $\nu\geq 1$,
\begin{align*}
    \normenu{\textbf{u}} &= |\textbf{u}_0| + 2\sum_{n=1}^{+\infty}|\textbf{u}_n|\nu^n = \sum_{n=0}^{+\infty} |\textbf{u}_n| \xi_n(\nu), \\
\intertext{where}
\xi_n(\nu) &= \left\lbrace \begin{array}{ll} 1 ,& \text{ if } n=0, \\ 2\nu^n ,& \text{ otherwise.} \end{array} \right.
\intertext{The Banach space we choose is}
\ell^1_\nu &= \left\lbrace \textbf{u} \in \mathbb{R}^\mathbb{N} \, , \, \normenu{\textbf{u}} < \infty \right\rbrace .
\end{align*}

This space represents smooth functions, as soon as $\nu>1$. There is a clear correspondence between the space $\ell^1_\nu$ and a functional space of analytic functions (see e.g.~\cite[p.98]{ZQ}). The Banach space $\ell^1_\nu$ is furthermore a (unital commutative) Banach algebra with the ``product'' $*$, the usual discretee convolution,
\begin{align*}
    (\textbf{u}*\textbf{v})_n &= \sum_{k\in\mathbb{Z}} \textbf{u}_{|k|}\textbf{v}_{|n-k|}, \\
    \intertext{and $\normenu{\cdot}$ is an algebraic norm,}
    \normenu{\textbf{u}*\textbf{v}} &\leq \normenu{\textbf{u}}\,\normenu{\textbf{v}}, \ \normenu{\textbf{1}}=1\, .
\end{align*}

\begin{rem}
    For $\nu < 1$, the space $\ell^1_\nu$ is no longer a Banach algebra.
\end{rem}

Let $L\in\mathcal{L}(\ell^1_\nu)$ a linear operator. It can be seen as a \emph{infinite matrix} $(L_{k,n})_{(k,n)\in\mathbb{N}^2}$ where \begin{equation*}
(L\textbf{u})_k = \sum_{n=0}^{+\infty}L_{k,n}\textbf{u}_n .\end{equation*}
The associated operator norm on $\mathcal{L}(\ell^1_\nu)$ can be expressed as follows,
\begin{equation*}
    \normenu{L} := \underset{\normenu{\textbf{u}} = 1}{\sup} \normenu{L\textbf{u}} =  \underset{n\in\mathbb{N}}{\sup} \dfrac{1}{\xi_n(\nu)} \normenu{L_{(\cdot,n)}}.
\end{equation*}
$L_{(\cdot,n)}$ corresponds to the $n$-th column vector of $L$, from a matrix point of view.

We denote by $\convomat{\textbf{u}}\in\mathcal{L}(\ell^1_\nu)$ the multiplication operator associated to the convolution by $\textbf{u}$, i.e.
\begin{align*}
     \convomat{\textbf{u}} \ :\  \ell^1_\nu &\longrightarrow \ell^1_\nu \\
     \textbf{v} &\longmapsto \textbf{u}*\textbf{v} 
\end{align*}
Notice that $\normenu{\convomat{\textbf{u}}} = \normenu{\textbf{u}}$, and that $ \|u\|_{\mathcal{C}^0} := \underset{x\in[a,b]}{\sup}|u(x)| \leq \normenu{\textbf{u}} $.

We also define the norm of a vector $(\textbf{u},\textbf{v})\in \ell^1_\nu \times \ell^1_\nu$:
\begin{equation*}
    \normenu{(\textbf{u},\textbf{v})} = \normenu{\textbf{u}} + \normenu{\textbf{v}} .
\end{equation*}
In the sequel, we also have to deal with linear operators $L$ acting on $\ell^1_\nu \times \ell^1_\nu$, for which we use a block-notation:
\begin{align*}
    L = \left( \begin{array}{c|c}
    L^{11} & L^{12}\\ \hline
    L^{21} & L^{22}
\end{array} \right), \text{ with }
    L^{ij} \in \mathcal{L}(\ell^1_\nu), \ i,j \in \{ 1,2 \}.
\end{align*}
The corresponding operator norm can then be expressed as follows:
\begin{align*}
    \normenu{L}&= \underset{ (\ell^1_\nu)^2\ni (\textbf{u},\textbf{v})\neq (\textbf{0},\textbf{0})}{\sup} \dfrac{\normenu{L^{11}\textbf{u}+L^{12}\textbf{v}} + \normenu{L^{21}\textbf{u}+L^{22}\textbf{v}}}{\normenu{\textbf{u}}+\normenu{\textbf{v}}} \\
    &=\sup_{n\in\mathbb{N}} \dfrac{1}{\xi_k(\nu)} \max \Big( \normenu{L^{11}_{(\cdot,n)}}+\normenu{L^{21}_{(\cdot,n)}}, \, \normenu{L^{12}_{(\cdot,n)}}+\normenu{L^{22}_{(\cdot,n)}} \Big).
\end{align*}

We need the Laplacian $\Delta$ and its inverse $\Delta^{-1}$ on Fourier sequences. The sequences $\Delta \textbf{v}$ and $\Delta^{-1}\textbf{v}$ are defined as follow :
\begin{align*}
    &\left(\Delta \textbf{v}\right)_n = \left(-\left(\dfrac{n\pi}{b-a}\right)^2\textbf{v}_n\right)_n,\quad n\geq 0, \\
    &\left(\Delta^{-1}\textbf{v}\right)_n = \left\lbrace \begin{array}{ll}
    0, &n=0, \\
    -\left(\dfrac{b-a}{n\pi}\right)^2\textbf{v}_n, &n\geq 1. \end{array}\right.
\end{align*}

\hfill

Finally, we will slightly abuse notation and denote by $\gamma(\textbf{v})$ the sequence of Fourier's coefficients of the function $\gamma\circ v$. In other words, $\gamma(\textbf{v}) = \mathcal{F}(\gamma \circ \mathcal{F}^{-1}(\textbf{v}))$. We insist on this way of looking at $\gamma$ to show that we want to express $\gamma(\textbf{v})$ in terms of $\textbf{v}$. This is the whole purpose of Section~\ref{sec:tool}.

\begin{rem}
    We do not need to assume a priori that $\gamma(\textbf{v})$ is well defined for an arbitrary element $\textbf{v}$ in $\ell^1_\nu$, but we will obtain as a byproduct of our proof that $\gamma(\textbf{v})$ is indeed well defined, and belongs to $\ell^1_\nu$, for all $\textbf{v}$ in a neighborhood of the approximate solutions that have been validated.
\end{rem}

\subsection{From the PDE to the zero-finding problem}
We use our notations and operations to rewrite the steady state problem associated to system \eqref{eq:syst_para}, namely
\begin{equation} \label{eq:syststat}
    \left\{\begin{array}{l l l l r}
        \Delta(\gamma(v)u) & + & \sigma u (1-u) &= 0,& \quad \text{ in } (a,b), \\
        d\Delta v & + & u-v &= 0,& \quad \text{ in } (a,b), \\
        \partial_n u = 0 ,  & &\partial_n v = 0,& & \text{ on } \{a,b\}. \\
    \end{array} \right .
\end{equation}
Thanks to the Fourier's transform, we can write directly the system \eqref{eq:syststat} in Fourier space: 
\begin{equation}\label{eq:Fourier1}
    \left\{\begin{array}{llll}
        \Delta \left(\gamma(\textbf{v)}*\textbf{u}\right) &+& \sigma \textbf{u}* (\textbf{1}-\textbf{u}) &= \textbf{0}, \\
        d\Delta \textbf{v} &+& \textbf{u}-\textbf{v} &= \textbf{0}.
    \end{array} \right .
\end{equation}
Then, let us denote
\begin{align*}
    & \textbf{U} = \begin{pmatrix} \textbf{u} \\ \textbf{v} \end{pmatrix}, \quad
    F(\textbf{U}) = \begin{pmatrix} \Delta & 0 \\ 0 & \Delta \end{pmatrix} \Phi(\textbf{U}) + R(\textbf{U}),
\intertext{with}
    & \Phi\begin{pmatrix} \textbf{u} \\ \textbf{v} \end{pmatrix} = \begin{pmatrix} \gamma (\textbf{v}) * \textbf{u} \\ d \textbf{v} \end{pmatrix}, \quad
    R\begin{pmatrix} \textbf{u} \\ \textbf{v} \end{pmatrix} = \begin{pmatrix} \sigma(\textbf{u}-\textbf{u}*\textbf{u}) \\ \textbf{u}-\textbf{v} \end{pmatrix}.
\end{align*} 
Thus, the system \eqref{eq:Fourier1} can be written as
\begin{equation}\label{eq:Fourier2}
    F(\textbf{U}) = \textbf{0} \,.
\end{equation}

So, with this formalism and with $\nu > 1$, a $\ell^1_\nu$-solution $\textbf{U}=(\textbf{u}, \textbf{v})$ of \eqref{eq:Fourier2} gives a $C^{\infty}$-solution of \eqref{eq:syststat}.

\subsection{Newton-Kantorovitch Theorem}

We have to solve the system $F(\textbf{U}) = 0$.
In order to prove the existence of a zero $F$, we use a fixed point method based on the Newton-Kantorovitch Theorem, for which we follow~\cite{B22}. There are other versions and applications of the Newton-Kantorovicth Theorem, see for instance \cite{AriKochTer05,CalRap97,DayLesMis07,Plu01,BreLesVen21,Yam98}.

Let $N \in \mathbb{N}$ and $\overline{\textbf{U}} = (\bar{\textbf{u}}, \bar{\textbf{v}}) \in (\ell^1_\nu)^2$ be an approximation of a stationary solution, this is our starting point for the fixed point method. Practically, $\overline{\textbf{U}}$ is a finite sequence of Fourier coefficients obtained numerically. Let $A$ be an approximate inverse of $DF(\overline{\textbf{U}})$, the Fréchet derivative of $F$ at $\overline{\textbf{U}}$. It will be explicitly built and discussed in Section~\ref{sec:bounds}.

\begin{thm}[Newton-Kantorovich Theorem]\label{th:NK} \emph{(See \cite[Theorem 2.5]{B22}).} 
With the notations introduced in this section, $r^*>0$ and $\nu>1$, assume there exist constants $Y,\, Z_1,\,\text{and } Z_2 $ satisfying :
\begin{subequations} \label{boundall}
\begin{align}
    &\normenu{AF(\myol{\textbf{U}})} \leq Y, \label{boundY} \\ 
    &\normenu{I - ADF(\myol{\textbf{U}})} \leq Z_1, \label{boundZ1} \\ 
    &\normenu{AD^2F(\textbf{U})} \leq Z_2, \quad \forall \textbf{U}\in\mathcal{B}_\nu(\myol{\textbf{U}},r^*),  \label{boundZ2}
\end{align}
\end{subequations}
and
\begin{subequations} \label{condNK}
\begin{align}
    &Z_1 < 1, \label{condNKa} \\ 
    &2YZ_2 < (1-Z_1)^2. \label{condNKb}
\end{align}
\end{subequations}
Then, for any $r$ satisfying
\begin{equation} \label{condNKc}
\dfrac{1-Z_1-\sqrt{(1-Z_1)^2-2YZ_2}}{Z_2} \leq r < \min(r^*, \dfrac{1-Z_1}{Z_2})\, ,
\end{equation}
there exists a unique fixed-point $\textbf{U}^*$ of $\, \textbf{U} \longmapsto I-AF(\textbf{U})$ in $\mathcal{B}_\nu(\myol{\textbf{U}},r)$ the closed ball of radius $r$ centered at $\myol{\textbf{U}}$ in $\ell^1_\nu \times \ell^1_\nu$. If moreover, $A$ is injective, then $\textbf{U}^*$ is a solution of $F(\textbf{U}) = 0$. The functions $(u^*,v^*)$, described by the vector of Fourier coefficients $(\textbf{u}^*,\textbf{v}^*)$, are then smooth solutions of \eqref{eq:syststat}.
\end{thm}
Thus, $\textbf{U}^*$ is a theoretical solution of our problem, and we have a numerical description $\overline{\textbf{U}}$ with a known margin of error. The proof of this theorem, for the specific choice of $A$ described in Section~\ref{sec:bounds}, can be found in \cite{B22}. 

Following the statement of Theorem~\ref{th:NK}, our work in the next sections will be to exhibit bounds $Y$, $Z_1$ and $Z_2$. The expression of each bound must be computable, and accurate enough to satisfy the conditions~\eqref{condNK} of Theorem~\ref{th:NK}. In Section~\ref{sec:tool}, we describe the tools and initial results needed to that end. In Section~\ref{sec:bounds}, we then derive each of the bounds satisfying conditions~\eqref{boundall}.

\section{Approximation tools and technical lemmas} \label{sec:tool}
We now have to work out how we fit into the assumptions and statements of Theorem~\ref{th:NK}. In this section, we introduce the notations and the concepts for matching theoretical and numerical objects. Then, to verify the hypotheses of the Newton-Kantorovitch Theroem~\ref{th:NK}, we need to explore those ideas and compute some preliminary results. In particular, we have to control the terms linked to $\gamma$ and show how $\gamma$ is seen in $\ell^1_\nu$.

\subsection{Notations for the computations of approximations}
\label{sec:notations_approx}

The following notations are motivated by the finiteness of the computer, which can only handle finite sequences. 
Let $N \in \mathbb{N}$, we denote by $\ell^1_{\nu,N}$ the subspace that contains finite sequences of size $N$. 
More precisely, $\textbf{u}\in\ell^1_\nu$ belongs to $\ell^1_{\nu,N}$ if and only if $\textbf{u}_n = 0$ for all $n>N$. 
For any $\textbf{u}$ in $\ell^1_\nu$, we denote by $\vectrunc[\textbf{u}]{N}$ its truncation in $\ell^1_{\nu,N}$, that is,
\begin{align*}
    (\vectrunc[\textbf{u}]{N})_n = \left\{
    \begin{aligned}
        &\textbf{u}_n \qquad &n\leq N,\\
        &0 \qquad &n>N.
    \end{aligned}
    \right.
\end{align*}
Since $\ell^1_{\nu,N}$ is not stable by convolution, we will have to adapt $N$ when we need to remove truncation errors, using that 
\begin{equation*}
    \textbf{u}\in\ell^1_{\nu,N_1},\ \textbf{v}\in\ell^1_{\nu,N_2} \implies \textbf{u}*\textbf{v} \in \ell^1_{\nu,N_1+N_2}. 
\end{equation*}
\begin{rem}
    We will keep in mind that $\ell^1_{\nu,N}$ is isomorphic to $\mathbb{R}^{N+1}$, indeed we start counting from $0$ in the sequences.
\end{rem}

In practice, while it is fine to use $\ell^1_{\nu,N}$ spaces for constructing approximate solutions, when conducting the proof we will have to rigorously control quantities like $\gamma(\textbf{v})$, which, for typical $\gamma$'s of the form~\eqref{eq:examples_gamma}, will no belong to any $\ell^1_{\nu,N}$ space even if $\textbf{v}$ does. In principle, the truncation operator introduced above provides a natural way to approximate an arbitrary element $\textbf{v}$ of $\ell^1_\nu$ by one which can be stored and manipulated on the computer, but in practice it requires exact knowledge of the first coefficients $\gamma(\textbf{v})_n$ for $n\leq N$, and an explicit error bound for the truncation, which are not always easy to obtain. 
Therefore, we will introduce different ways of practically approximating elements of $\ell^1_\nu$, and specifically those of the form $\gamma(\textbf{v})$, for a class of functions $\gamma$, and assuming $\textbf{v}$ is an element of $\ell^1_\nu$ on which we have some control. In the sequel, whenever $\textbf{v}$ is an element of $\ell^1_\nu$, $\vecapp{\textbf{v}}$ denotes an element of $\ell^1_{\nu,N}$ (for some $N$ which we will not always specify), which is meant to approximate $\textbf{v}$. We also denote by $\vecdiff{\textbf{v}}$ the corresponding error, namely $\vecdiff{\textbf{v}} = \textbf{v}-\vecapp{\textbf{v}}$, and by $\vecbound{\textbf{v}}$ an explicit error bound, i.e. a computable number such that  $\normenu{\vecdiff{\textbf{v}}} \leq \vecbound{\textbf{v}}$. Practically, we adapt our choice of $\vecapp{\textbf{v}}$ to the problem at hand, as described in the remainder of this section.
\begin{rem} \label{rem:trunc}
In practice, depending on the procedure we use to obtain an approximation $\vecapp{\textbf{v}}$ and an error bound $\vecbound{\textbf{v}}$, the $N$ for which $\vecapp{\textbf{v}}$ belongs to $\ell^1_{\nu,N}$ might be quite large. This can be detrimental if we are to re-use $\vecapp{\textbf{v}}$ in other computations. However, we can easily replace $\vecapp{\textbf{v}}$ by another approximation having a truncation level $N'<N$:
\begin{align*}
    \textbf{v} = \vectrunc[\vecapp{\textbf{v}}]{N'} + \vecapp{\textbf{v}}-\vectrunc[\vecapp{\textbf{v}}]{N'} + \vecdiff{\textbf{v}},
\end{align*}
where $\vectrunc[\vecapp{\textbf{v}}]{N'}$ becomes our new approximation, and the new error can be bounded by
\begin{align*}
    \underbrace{\normenu{\vecapp{\textbf{v}}-\vectrunc[\vecapp{\textbf{v}}]{N'}}}_{\text{computable}} + \vecbound{\textbf{v}}.
\end{align*}
\end{rem}
We extend the notion of truncation to operators on $\mathcal{L}(\ell^1_\nu)$.
Given an operator $L$ in $\mathcal{L}(\ell^1_\nu)$, we denote by $\vectrunc[L]{N}$ its restriction to $\ell^1_{\nu,N}$, both in terms of domain and co-domain. That is, for any $\textbf{v}$ in $\ell^1_{\nu,N}$, $\left(\vectrunc[L]{N}\right) \textbf{v} = \vectrunc[\left(L\textbf{v}\right)]{N}$. 

The restriction to $N$ extends to the operators of $\mathcal{L}(\ell^1_\nu \times \ell^1_\nu)$ in a natural way, i.e.
\begin{align*}
    \vectrunc[L]{N} = \left( \begin{array}{c|c}
    \vectrunc[L^{11}]{N} & \vectrunc[L^{11}]{N}\\ \hline
    \vectrunc[L^{21}]{N} & \vectrunc[L^{22}]{N}
\end{array} \right).
\end{align*}

\subsection{How to compute in the Banach algebra with errors ?} \label{sec:compute}

Let us now focus on the main difficulty of this paper, namely controlling $\gamma(\textbf{v})$ for functions $\gamma$ of the form: $\gamma(x) = \dfrac{f(x)}{g(x)}$, $x \in \Omega \subset \mathbb{R}$. We assume that $f$ and $g$ are analytic, but $g$ may vanish. We propose some methods to determine an approximate element $\vecapp{\gamma(\textbf{v})}$ of $\gamma(\textbf{v})$, and a corresponding error bound $\vecbound{\gamma(\textbf{v})}$, for any $\textbf{v}$ belonging to $\ell^1_\nu$ with a finite approximation $\vecapp{\textbf{v}}\in\ell^1_{\nu,N}$ and an associated error bound $\vecbound{\textbf{v}}$.
We will use the propositions below in two special ways in the context of Theorem~\ref{th:NK}, where we choose $\vecapp{\textbf{v}}=\bar{\textbf{v}}$, a numerical solution of \eqref{eq:syst_para}: 
\begin{itemize}
    \item Assuming $\vecbound{\textbf{v}}=0$, we obtain a description of an approximate element $\vecapp{\gamma(\vecapp{\textbf{v}})}$ and an error $\vecbound{\gamma(\vecapp{\textbf{v}})}$ due to the operations we apply.
    \item Assuming $\vecbound{\textbf{v}}=r^*$ (from Theorem~\ref{th:NK}), we control the element $\gamma(\textbf{v})$ described in $\ell^1_\nu$ by a finite element $\vecapp{\gamma(\vecapp{\textbf{v}})}$ and an error $\vecbound{\gamma(\textbf{v})}$ that is provided by $\gamma$.
\end{itemize}
All this will be useful to determine the bounds $Y$, $Z_1$ and $Z_2$ in the hypotheses~\eqref{boundall} of Theorem~\ref{th:NK}. We will give some examples of our results on entire functions, rational fractions, and compositions thereof.

\subsubsection{Division and product}
We introduce those two operations in our Banach algebra $\ell^1_\nu$. We specify the bounds in each case and how we can control the errors of the result of the operation.
 
\begin{lem}  \label{lem:approxinv}
Let $\textbf{z} \in \ell^1_\nu$ and $\textbf{a} \in \ell^1_\nu$ such that $\normenu{\textbf{a}*\textbf{z}-\textbf{1}} < 1$. Then, $\textbf{z}$ is invertible and we have
\begin{align*}
    \textbf{z}^{-1} &=  \textbf{a}*\sum_{j=0}^{+\infty} (-1)^j(\textbf{a}*\textbf{z} - \textbf{1})^{*j}, \\
    \normenu{\textbf{z}^{-1}-\textbf{a}} &\leq \normenu{\textbf{a}} \dfrac{\normenu{\textbf{a}*\textbf{z}-\textbf{1}}}{1-\normenu{\textbf{a}*\textbf{z}-\textbf{1}}}.
\end{align*}
\end{lem}
\begin{proof} The existence of the inverse is guaranteed by the criterion of the Neumann series. Writting down the difference and applying the triangle inequality gives the result above.
\end{proof}

\begin{rem}
    In practice, $\textbf{a}$ should of course be taken as a numerically computed approximation of $\textbf{z}^{-1}$ and will therefore be finite. If $\textbf{z}$ is itself finite, then $\normenu{\textbf{a}*\textbf{z}-\textbf{1}}$ amounts to a finite computation and can therefore be evaluated exactly (up to rounding errors, which are taken care of using interval arithmetic~\cite{Moo79,Tuc11}).
    If not, as long as we have a decomposition $\textbf{z} = \vecapp{\textbf{z}} + \vecdiff{\textbf{z}}$ with a computable error bound $\vecbound{\textbf{z}}$ for $\vecdiff{\textbf{z}}$, we can use that
    \begin{align*}
        \normenu{\textbf{a}*\textbf{z}-\textbf{1}} \leq \normenu{\textbf{a}*\vecapp{\textbf{z}} - \textbf{1}} + \normenu{\textbf{a}} \vecbound{\textbf{z}},
    \end{align*}
    the right-hand side now being computable.
\end{rem}

\begin{cor}\label{cor:inv}
    Let us take any $\textbf{z}\in \ell^1_\nu$, let $N\in\mathbb{N}$ and $\vecapp{\textbf{z}}\in \ell^1_{\nu,N}$ an approximation of $\textbf{z}$ with error bound $\vecbound{\textbf{z}}\geq 0$.
    Let $\textbf{a} \in \ell^1_{\nu,N}$ such that $\normenu{\textbf{a}*\vecapp{\textbf{z}} - \textbf{1}} + \normenu{\textbf{a}} \vecbound{\textbf{z}} < 1$.

    Then, $\textbf{z}$ is invertible in $\ell^1_\nu$ and we can take 
    \begin{align*}
    \vecapp{\textbf{z}^{-1}} &= \textbf{a} \\
    \vecbound{\textbf{z}^{-1}} &= \normenu{\textbf{a}} \dfrac{\normenu{\textbf{a}*\vecapp{\textbf{z}} - \textbf{1}} + \normenu{\textbf{a}} \vecbound{\textbf{z}}}{1 - \normenu{\textbf{a}*\vecapp{\textbf{z}} - \textbf{1}} - \normenu{\textbf{a}} \vecbound{\textbf{z}}}.
    \end{align*}
\end{cor}

\begin{proof}
The function $x \mapsto \dfrac{x}{1-x}$ is increasing and we have $\normenu{\textbf{a}*\textbf{z} - \textbf{1}} \leq \normenu{\textbf{a}*\vecapp{\textbf{z}} - \textbf{1}} + \normenu{\textbf{a}}\vecbound{\textbf{z}}$. Hence, we have the expected result. 
\end{proof}

\begin{rem}
  Lemma~\ref{lem:approxinv} shows how to compute naturally in the Banach algebra. Corollary~\ref{cor:inv} shows how we get an explicit expression of an error bound depending on the known errors.  Even if the Corollary contains the Lemma, the two results express two ways of thinking.
\end{rem}

\begin{lem}\label{lem:prod}
    Let $\textbf{x}, \, \textbf{y} \, \in \ell^1_\nu$, $N\in\mathbb{N}$. Let $\vecapp{\textbf{x}}, \, \vecapp{\textbf{y}} \in \ell^1_{\nu,N}$ be approximations and $\vecbound{\textbf{x}}, \, \vecbound{\textbf{y}} \geq 0$ be error bounds of $\textbf{x}$ and $\textbf{y}$ respectively. We can take
    \begin{align*}
        \vecapp{\textbf{x}*\textbf{y}} &= \vecapp{\textbf{x}}*\vecapp{\textbf{y}}, \\
        \vecbound{\textbf{x}*\textbf{y}} &= \normenu{\vecapp{\textbf{x}}}\vecbound{\textbf{y}} + \normenu{\vecapp{\textbf{y}}}\vecbound{\textbf{x}} + \vecbound{\textbf{x}}\vecbound{\textbf{y}}.
    \end{align*}
\end{lem}

Combining Corollary~\ref{cor:inv} with the straightforward statement above, we have all the required ingredients to understand Banach algebra calculus with error handling.

\begin{thm}\label{th:quotient}
Let $\textbf{x}, \, \textbf{y} \, \in \ell^1_\nu$. Let $N\in\mathbb{N}$, $\vecapp{\textbf{x}}, \, \vecapp{\textbf{y}} \in \ell^1_{\nu,N}$ be approximations and $\vecbound{\textbf{x}}, \, \vecbound{\textbf{y}} \geq 0$ be error bounds of $\textbf{x}$ and $\textbf{y}$ respectively. Let $\textbf{a}\in\ell^1_{\nu,N}$ such that $\normenu{\textbf{a}*\vecapp{\textbf{y}} - \textbf{1}} + \normenu{\textbf{a}} \vecbound{\textbf{y}} < 1$. We can take
\begin{align*}
    \vecapp{\textbf{x}*\textbf{y}^{-1}} &= \vecapp{\textbf{x}}*\textbf{a}, \\
    \vecbound{\textbf{x}*\textbf{y}^{-1}} &= \normenu{\textbf{a}} \dfrac{\normenu{\vecapp{\textbf{x}}}\left(\normenu{\textbf{a}*\vecapp{\textbf{y}} - \textbf{1}} + \normenu{\textbf{a}} \vecbound{\textbf{y}}\right) +\vecbound{\textbf{x}}}{1-\normenu{\textbf{a}*\vecapp{\textbf{y}} - \textbf{1}} - \normenu{\textbf{a}} \vecbound{\textbf{y}}}.
\end{align*}

\end{thm}

\begin{proof}
Apply Corollary~\ref{cor:inv} to $\textbf{z} = \textbf{y} ,\, \vecapp{(\vecapp{\textbf{y}}^{-1})} = \textbf{a} $. 
Then $\textbf{y}$ is invertible and belongs to $\ell^1_\nu$. We have
$\vecapp{(\textbf{y}^{-1})} = \textbf{a}$ and $\vecbound{\textbf{y}^{-1}}$ associated.

Then we choose, according to Lemma~\ref{lem:prod}
    \begin{align*}
        \vecapp{\textbf{x}*\textbf{y}^{-1}} &= \vecapp{\textbf{x}}*\textbf{a}, \\
        \vecbound{\textbf{x}*\textbf{y}^{-1}} &= \normenu{\vecapp{\textbf{x}}}\vecbound{\textbf{y}^{-1}} + \normenu{\textbf{a}}\vecbound{\textbf{x}} + \vecbound{\textbf{y}^{-1}}\vecbound{\textbf{x}}.
    \end{align*}
We simplify and obtain the result.
\end{proof}

\subsubsection{Analytic functions} \label{sec:analytic_functions}

In the previous subsection, we have seen how to deal with fractions in the Banach algebra $\ell^1_\nu$, with rigorous error control, provided we had a description of both the numerator and the denominator in terms of an approximation and an error bound. In this subsection, we discuss how such approximations and error bounds for the numerator and the denominator can be obtained, provided those are analytic functions.

Let $f$ be an analytic function.
We consider $\textbf{v}\in\ell^1_\nu$ and for $N\in\mathbb{N}$, $\vecapp{\textbf{v}}\in\ell^1_{\nu,N}$ an approximation of $\textbf{v}$ and $\vecbound{\textbf{v}}\geq 0$ an associated error.
We will first use the analytic description of $f$ to obtain an explicit expression of $\vecapp{f(\vecapp{\textbf{v}})}$ and $\vecbound{f(\vecapp{\textbf{v}})}$.
Secondly, we will use the analytic properties of $f$ to get an explicit bound $\vecbound{f(\textbf{v})}$ on the norm of $\vecdiff{f(\textbf{v})} = f(\textbf{v}) - \vecapp{f(\vecapp{\textbf{v}})}$.
This means that we choose, as an approximation of $f(\textbf{v})$, a nearby element $\vecapp{f(\vecapp{\textbf{v}})}$. It belongs to some $\ell^1_{\nu,N'}, \ N' \in \mathbb{N}$, with a computable description.

Since we work in a (commutative) Banach algebra, we can still make use of Taylor expansions (see e.g.~\cite[Chapter 12]{FuncAn}), and the Fourier transform carries all the calculations from $\mathbb{R}$ to $\ell^1_\nu$ in a straightforward way.

\begin{thm} \label{th:Taylor}
Let $K\in\mathbb{N}\backslash \lbrace 0 \rbrace$ and  $f$ a function of class $\mathcal{C}^K(\mathbb{R};\,\mathbb{R})$.
Then, the function defined by \begin{equation*}f : \left\lbrace\begin{array}{lll}
    \ell^1_\nu &\longrightarrow & \ell^1_\nu   \\
    \textbf{v} &\longmapsto & \mathcal{F}(f\circ\mathcal{F}^{-1}(\textbf{v})) := f(\textbf{v})
\end{array} \right. \text{is $K$-differentiable.}
\end{equation*} 

We have the following Taylor expansion. Let $\textbf{v}_0\in\ell^1_\nu$,
\begin{equation*}
    \forall \textbf{v}\in\ell^1_\nu, \, \normenu{f(\textbf{v}) - \sum_{k=0}^{K-1} \dfrac{1}{k!}f^{(k)}(\textbf{v}_0)*(\textbf{v}-\textbf{v}_0)^{* k}} \leq \dfrac{\normenu{\textbf{v}-\textbf{v}_0}^K}{K!} \underset{\textbf{z}\in[\textbf{v}_0,\textbf{v}]}{\sup} \normenu{f^{(K)}(\textbf{z})}.
\end{equation*}
\end{thm}

We can apply this result to any entire function and go even further.

\begin{prop} \label{prop:entire_func}
    Let $f$ be an entire function, written $f(x) = \displaystyle \sum_{k=0}^{+\infty} a_k x^k $. Moreover, we denote $|f'|$ the entire function defined as $|f'|(x) = \displaystyle \sum_{k=0}^{+\infty} k|a_k| x^{k-1}$. 
    Let $\textbf{v}\in\ell^1_\nu$ and $N\in\mathbb{N}$. Let $\vecapp{\textbf{v}}\in\ell^1_{\nu,N}$ an approximation and $\vecbound{\textbf{v}}\geq 0$ an error bound.
    Taking
    \begin{align*}
        \vecapp{f(\vecapp{\textbf{v}})} &= \displaystyle \sum_{k=0}^{K-1} a_k \vecapp{\textbf{v}}^{*k}, \\
    \intertext{we get the following error bound}
        \vecbound{f(\vecapp{\textbf{v}})} &= \dfrac{\normenu{\vecapp{\textbf{v}}}^K}{K!} \underset{\textbf{z}\in[\textbf{0},\vecapp{\textbf{v}}]}{\sup} \normenu{f^{(K)}(\textbf{z})}.
    \end{align*}
Furthermore, we can take
    \begin{align*}
        \vecapp{f(\textbf{v})} &= \vecapp{f(\vecapp{\textbf{v}})} \\
         \vecbound{f(\textbf{v})} &= \vecbound{f(\vecapp{\textbf{v}})} +  |f'|(\normenu{\vecapp{\textbf{v}}}+\vecbound{\textbf{v}}) \vecbound{\textbf{v}} .
    \end{align*}
\end{prop}

\begin{proof}
    We just write $f(\textbf{v}) = \vecapp{f(\vecapp{\textbf{v}})} + \vecdiff{f(\vecapp{\textbf{v}})} + f(\textbf{v}) - f(\vecapp{\textbf{v}})$ and apply the triangular inequality. We have $\normenu{\vecdiff{f(\vecapp{\textbf{v}})}}\leq \vecbound{f(\vecapp{\textbf{v}})}$ thanks to Theorem~\ref{th:Taylor}.
    The last thing to bound is $\normenu{f(\textbf{v}) - f(\vecapp{\textbf{v}})}$. To do this we apply Theorem~\ref{th:Taylor} with $K=1$, $\textbf{v}_0=\vecapp{\textbf{v}}$. We have
    \begin{align*}
        \normenu{f(\textbf{v}) - f(\vecapp{\textbf{v}})} &\leq \sup_{\textbf{z}\in [\textbf{v},\vecapp{\textbf{v}}]} \normenu{f'(\textbf{z})} \normenu{\textbf{v}-\vecapp{\textbf{v}}} \\
        &\leq \sup_{\textbf{z}\in \mathcal{B}_\nu(\vecapp{\textbf{v}},\vecbound{\textbf{v}})} \normenu{f'(\textbf{z})} \vecbound{\textbf{v}} \\
        &\leq \sup_{\textbf{z}\in \mathcal{B}_\nu(\vecapp{\textbf{v}},\vecbound{\textbf{v}})} \normenu{\sum_{k=0}^{+\infty} k\,a_k \textbf{z}^{*k-1}} \vecbound{\textbf{v}} \\
        &\leq \sup_{\textbf{z}\in \mathcal{B}_\nu(\vecapp{\textbf{v}},\vecbound{\textbf{v}})} \sum_{k=0}^{+\infty} k\,|a_k| \normenu{\textbf{z}}^{k-1} \vecbound{\textbf{v}},
\end{align*}
but $\normenu{\textbf{z}}\leq \normenu{\vecapp{\textbf{v}}}+\vecbound{\textbf{v}}$ for any $\textbf{z}\in \mathcal{B}_\nu(\vecapp{\textbf{v}},\vecbound{\textbf{v}})$ then,
\begin{align*}
        \normenu{f(\textbf{v}) - f(\vecapp{\textbf{v}})} &\leq \sum_{k=0}^{+\infty} k\,|a_k| (\normenu{\vecapp{\textbf{v}}}+\vecbound{\textbf{v}})^{k-1} \vecbound{\textbf{v}} \\
        &\leq |f'|(\normenu{\vecapp{\textbf{v}}}+\vecbound{\textbf{v}})\vecbound{\textbf{v}}.
    \end{align*}
Hence, the result holds.
\end{proof}

\begin{rem}
    In practice, we need to be able to evaluate $\vecbound{f(\vecapp{\textbf{v}})}$ and $\vecbound{f(\textbf{v})}$ (or at least to get computable upper-bounds for them). This must be done on a case by case basis, depending on the function $f$ and the available information we have on its derivatives (or on the asymptotic behavior of the coefficients $a_k$), and might prove difficult for some analytic functions. However, among the functions $\gamma$ that are usually considered for~\eqref{eq:syst_para}, the typical entire nonlinearities are based on the exponential function, for which these computations can be carried out. See Section~\ref{sec:examples} for some explicit examples.
\end{rem}

\begin{rem} 
\label{rem:entire_analytic}
We stated Proposition~\ref{prop:entire_func} for entire functions for the sake of convenience, but the result generalizes in a straightforward way to functions $f$ which are merely analytic, provided the disk of center $0$ and radius $\normenu{\vecapp{\textbf{v}}}+\vecbound{\textbf{v}}$ is contained in the domain of analyticity of $f$. A slightly more natural and less restrictive assumption would be to require that the range of $\vecapp{\textbf{v}}$ is contained in the domain of analyticity of $f$, and at a distance at least $\vecbound{\textbf{v}}$ of its boundary. This goes beyond the scope of this work, but will be the subject or further studies.
\end{rem}

\begin{rem} \label{rem:pol}
    When $f$ is a polynomial $P = \displaystyle \sum_{k=0}^{\deg P}a_k X^k$ we have the following result as a special case of Proposition~\ref{prop:entire_func}:
    \begin{align*}
        \vecapp{P(\vecapp{\textbf{v}})} &= \sum_{k=0}^{\deg P}a_k \vecapp{\textbf{v}}^{*k}, \\
        \vecbound{P(\vecapp{\textbf{v}})} &= 0, \\
        \vecapp{P(\textbf{v})} &= \vecapp{P(\vecapp{\textbf{v}})}, \\
        \vecbound{P(\textbf{v})} &= |P'|(\normenu{\vecapp{\textbf{v}}}+\vecbound{\textbf{v}}) \vecbound{\textbf{v}}.
    \end{align*}
\end{rem}

\subsubsection{Examples}
\label{sec:examples}

In this whole section, we again take $\textbf{v}\in\ell^1_\nu$ , $N\in\mathbb{N}$, $\vecapp{\textbf{v}}\in\ell^1_{\nu,N}$ an approximation and $\vecbound{\textbf{v}}\geq 0$ an error bound.
We will show the choices we made to get a description in $\ell^1_\nu$ of the functions $\gamma$ involved in Theorem~\ref{th:single_sol}, Theorem~\ref{th:WX} and Theorem~\ref{th:bif_diag}.
\paragraph{\underline{The exponential}:}
Let us consider $f = \e{\alpha\,\cdot}$, for some $\alpha\in\R$. We can apply Proposition~\ref{prop:entire_func} and use the fact that $\underset{\textbf{z}\in[\textbf{0},\vecapp{\textbf{v}}]}{\sup} \normenu{f^{(K)}(\textbf{z})} \leq |\alpha|^{K} \e{|\alpha|\, \normenu{\vecapp{\textbf{v}}}}$ to get
\begin{align*}
    \vecapp{f(\vecapp{\textbf{v}})} &= \displaystyle \sum_{k=0}^{K-1} \dfrac{\alpha^k}{k!} \vecapp{\textbf{v}}^{*k}, \\
    \vecbound{f(\vecapp{\textbf{v}})} &= \dfrac{|\alpha|^K\normenu{\vecapp{\textbf{v}}}^K}{K!} \e{|\alpha|\, \normenu{\vecapp{\textbf{v}}}}, 
\end{align*}
and
\begin{align*}
    \vecapp{f(\textbf{v})} &= \vecapp{f(\vecapp{\textbf{v}})}, \\ 
    \vecbound{f(\textbf{v})} &= \vecbound{f(\vecapp{\textbf{v}})} + |\alpha|\e{|\alpha|(\normenu{\vecapp{\textbf{v}}}+\vecbound{\textbf{v}})}\vecbound{\textbf{v}},
\end{align*}
where the error bounds are now computable. 
\begin{rem}
\label{rem:exp}
    Using some specific properties of the exponential function, we could also get a sharper error bound for $\normenu{f(\textbf{v})-\vecapp{f(\textbf{v})}}$. Indeed, when in the proof of Proposition~\ref{prop:entire_func} we need to estimate $\displaystyle\sup_{\textbf{z}\in \mathcal{B}_\nu(\vecapp{\textbf{v}},\vecbound{\textbf{v}})} \normenu{f'(\textbf{z})}$, we do have
    \begin{align*}
        \sup_{\textbf{v}\in\mathcal{B}_{\nu}(\vecapp{\textbf{v}},\vecbound{\textbf{v}})} \normenu{f'(\textbf{v})} \leq \left(\normenu{\vecapp{f'(\vecapp{\textbf{v}})}} + \vecbound{f'(\vecapp{\textbf{v}})}\right) e^{|\alpha|\vecbound{\textbf{v}}},
    \end{align*}
    which could be significantly smaller than $|\alpha|\e{|\alpha|(\normenu{\vecapp{\textbf{v}}}+\vecbound{\textbf{v}})}\vecbound{\textbf{v}}$.
\end{rem}

\paragraph{\underline{A fraction involving an exponential}:}
In the literature \cite{WX21}, the function $\gamma(v) = \dfrac{1}{1 + \e{9(v - 1)}}$ is considered in system~\eqref{eq:syst_para}. So let us manage it with our new tools. We first have to consider $g = 1 + \e{9( \ \cdot \, - 1)}$, which is not too far from the first example. We can choose,
\begin{align*} 
    \vecapp{g(\vecapp{\textbf{v}})} &= \textbf{1} + \sum_{k=0}^{K-1} \dfrac{9^k}{k!} (\vecapp{\textbf{v}}-\textbf{1})^{k}, \\ 
    \vecbound{g(\vecapp{\textbf{v}})} &= \dfrac{9^K\normenu{\vecapp{\textbf{v}}-\textbf{1}}^K}{K!} \e{9\normenu{\vecapp{\textbf{v}}-\textbf{1}}},
\end{align*}
and
\begin{align*}
    \vecapp{g(\textbf{v})} &= \vecapp{g(\vecapp{\textbf{v}})}, \\
    \vecbound{g(\textbf{v})} &= \vecbound{g(\vecapp{\textbf{v}})} + 9\e{9(\normenu{\vecapp{\textbf{v}}-\textbf{1}}+\vecbound{\textbf{v}})}\vecbound{\textbf{v}}.
\end{align*}
We point out that Remark~\ref{rem:exp} also applies here to get a better bound $\vecbound{g(\textbf{v})}$.

We can then apply Corollary~\ref{cor:inv}, with $\textbf{a}\in \ell^1_{\nu,N}$ such that $\normenu{\textbf{a}*\vecapp{g(\vecapp{\textbf{v}})} - \textbf{1}} + \normenu{\textbf{a}}\vecbound{g(\vecapp{\textbf{v}})}< 1$, both to $g(\vecapp{\textbf{v}})$ and to $g(\textbf{v})$, which yields,
\begin{align*}
    \vecapp{g(\vecapp{\textbf{v}})^{-1}} &= \textbf{a}, \\
    \vecbound{g(\vecapp{\textbf{v}})^{-1}} &= \normenu{\textbf{a}} \dfrac{\normenu{\textbf{a}*\vecapp{{g(\vecapp{\textbf{v}})}} - \textbf{1}} + \normenu{\textbf{a}} \vecbound{g(\vecapp{\textbf{v}})}}{1 - \normenu{\textbf{a}*\vecapp{g(\vecapp{\textbf{v}})} - \textbf{1}} - \normenu{\textbf{a}} \vecbound{g(\vecapp{\textbf{v}})}},
\end{align*}
and
\begin{align*}
    \vecapp{g(\textbf{v})^{-1}} &= \vecapp{g(\vecapp{\textbf{v}})^{-1}}, \\
    \vecbound{g(\textbf{v})^{-1}} &= \normenu{\textbf{a}}\dfrac{ \normenu{\textbf{a}*\vecapp{g(\vecapp{\textbf{v}})} - \textbf{1}} + \normenu{\textbf{a}}(\vecbound{g(\vecapp{\textbf{v}})} + 9\e{9(\normenu{\vecapp{\textbf{v}}-\textbf{1}}+\vecbound{\textbf{v}})}\vecbound{\textbf{v}})}{1 - \normenu{\textbf{a}*\vecapp{g(\vecapp{\textbf{v}})} - \textbf{1}} - \normenu{\textbf{a}}(\vecbound{g(\vecapp{\textbf{v}})} + 9\e{9(\normenu{\vecapp{\textbf{v}}-\textbf{1}}+\vecbound{\textbf{v}})}\vecbound{\textbf{v}})}.
\end{align*}

\paragraph{\underline{A rational fraction}:}
For a rational fraction $\gamma=\dfrac{P}{Q}$, we apply Remark~\ref{rem:pol} and Theorem~\ref{th:quotient}, with $\textbf{a}\in\ell^1_{\nu,N}$ such that $\normenu{\textbf{a}*Q(\vecapp{\textbf{v}}) - \textbf{1}} + \normenu{\textbf{a}}|Q'|(\normenu{\vecapp{\textbf{v}}}+\vecbound{\textbf{v}})\vecbound{\textbf{v}}<1$, which yields
\begin{align*}
        \vecapp{\gamma(\vecapp{\textbf{v}})} &= P(\vecapp{\textbf{v}})*\textbf{a}. \\
        \vecbound{\gamma(\vecapp{\textbf{v}})}
        &= \normenu{P(\vecapp{\textbf{v}})} \normenu{\textbf{a}}\dfrac{\normenu{\textbf{a}*Q(\vecapp{\textbf{v}})-\textbf{1}}}{1 - \normenu{\textbf{a}*Q(\vecapp{\textbf{v}})-\textbf{1}}},
\end{align*}
and
\begin{align*}
    \vecapp{\gamma(\textbf{v})} &= \vecapp{\gamma(\vecapp{\textbf{v}})}, \\
    \vecbound{\gamma(\textbf{v})} &= \normenu{\textbf{a}}\dfrac{\normenu{P(\vecapp{\textbf{v}})}\left(\normenu{\textbf{a}*Q(\vecapp{\textbf{v}})-\textbf{1}} + \normenu{\textbf{a}}|Q'|(\normenu{\vecapp{\textbf{v}}}+\vecbound{\textbf{v}})\vecbound{\textbf{v}}\right)+|P'|(\normenu{\vecapp{\textbf{v}}} +\vecbound{\textbf{v}})\vecbound{\textbf{v}}}{1 - \normenu{\textbf{a}*Q(\vecapp{\textbf{v}})-\textbf{1}} - \normenu{\textbf{a}}|Q'|(\normenu{\vecapp{\textbf{v}}}+\vecbound{\textbf{v}})\vecbound{\textbf{v}}}.
\end{align*}

\begin{rem}
    A rational fraction is a analytic function. Therefore, we might deal with it as discussed in Remark~\ref{rem:entire_analytic}. However, even in situations where this would prove feasible, we expect the approximation and estimates obtained via Theorem~\ref{th:quotient} to be both sharper and cheaper to compute.
\end{rem}

\section{Derivation of the bounds}\label{sec:bounds}

In this section, we first define the operator $A$ that we are going to use in Theorem~\ref{th:NK}, and then derive bounds $Y$, $Z_1$ and $Z_2$ satisfying assumptions \eqref{boundY}, \eqref{boundZ1}
and \eqref{boundZ2}. These estimates rely on the controlled approximations of $\gamma$ introduced in Section~\ref{sec:tool}. 
Let $\overline{\textbf{U}} = (\bar{\textbf{u}},\bar{\textbf{v}}) \in (\ell^1_{\nu,N})^2$, an approximate solution of \eqref{eq:Fourier2}. Thanks to Section~\ref{sec:tool} with $\vecapp{\textbf{v}}=\bar{\textbf{v}}$, we can calculate $\vecapp{\gamma(\bar{\textbf{v}})}$, $\vecbound{\gamma(\bar{\textbf{v}})}$, $\vecapp{\gamma'(\bar{\textbf{v}})}$, $\vecbound{\gamma'(\bar{\textbf{v}})}$,  $\vecapp{\gamma''(\bar{\textbf{v}})}$ and $\vecbound{\gamma''(\bar{\textbf{v}})}$, at least for the functions $\gamma$ used in Theorem~\ref{th:single_sol}, Theorem~\ref{th:WX} and Theorem~\ref{th:bif_diag}.

For any $\textbf{U} = (\textbf{u},\textbf{v})$ in $\ell^1_\nu$, we recall that
\begin{align*}
    F(\textbf{U}) &= \begin{pmatrix}\Delta(\gamma(\textbf{v})*\textbf{u})+\sigma\textbf{u}*(\textbf{1}-\textbf{u}) \\ d\Delta\textbf{v}+\textbf{u}-\textbf{v}\end{pmatrix}, \\
    \intertext{so that}
    DF(\textbf{U}) &= \begin{pmatrix} \Delta\gamma(\textbf{v})+\sigma(\textbf{1}-2\textbf{u}) & \Delta(\gamma'(\textbf{v})*\textbf{u}) \\ \textbf{1} & d\Delta-\textbf{1}\end{pmatrix}.
\end{align*}

For $A$, we consider a well chosen approximation of the inverse of $DF(\overline{\textbf{U}})$, combining a numerically approximated inverse of a truncation of $DF(\overline{\textbf{U}})$ and an infinite rest that correspond to the ``main'' part of the inverse. Following~\cite{B22}, we therefore take
\begin{equation*}
 A = \left( \begin{array}{c|c} 
\begin{array}{cc}
     \begin{array}{c|} A^{11}_{11} \\ \hline \end{array} &  \\
      & \mathrm{M}(\textbf{w}^{11})\Delta^{-1} 
\end{array} 
&
\begin{array}{cc}
     \begin{array}{c|} A^{12}_{11} \\ \hline \end{array}&  \\
      & \mathrm{M}(\textbf{w}^{12})\Delta^{-1}
\end{array} 
\\ \hline
\begin{array}{cc}
     \begin{array}{c|} A^{21}_{11} \\ \hline \end{array}&  \\
      & \mathrm{M}(\textbf{w}^{21})\Delta^{-1}  
\end{array}
&
\begin{array}{cc}
     \begin{array}{c|} A^{22}_{11} \\ \hline \end{array}&  \\
      & \mathrm{M}(\textbf{w}^{22})\Delta^{-1}
\end{array}
\end{array}\right) = \left(\begin{array}{c|c}
    A^{11} & A^{12}  \\ \hline
    A^{21} & A^{22}
\end{array} \right),
\end{equation*}
The block-by-block notation is inherited from the two species system. We use sub-blocks to separate the finite part and the tail of the operator. Each $A^{ij}_{11}$ is a linear operator on $\ell^1_{\nu,2N}$, this choice is explained in Remark~\ref{rem:size}, just below.

The blocks $A^{ij}_{11}$ are truncation of an inverted truncation of $DF(\overline{\textbf{U}})$, 
\begin{align*}
    \left(\begin{array}{c|c}
    A^{11}_{11} & A^{12}_{11}  \\ \hline
    A^{21}_{11} & A^{22}_{11}
\end{array} \right) \approx
\left[DF(\overline{\textbf{U}})|_{4N}\right]^{-1}\big|_{2N},
\end{align*}

and $\textbf{w}^{ij}$, $i,j\in{1,2}$, are elements of $\ell^1_{\nu,N}$ such that 
\begin{equation*}
\left\lbrace \begin{array}{ll}
\textbf{w}^{11}&\approx \gamma(\bar{\textbf{v}})^{-1}, \\
\textbf{w}^{12}&\approx -\dfrac{1}{d}\textbf{w}^{11}*(\gamma'(\bar{\textbf{v}})*\bar{\textbf{u}}),\\
\textbf{w}^{21}&= \textbf{0},\\
\textbf{w}^{22}&=\dfrac{1}{d}\textbf{1}.
\end{array} \right.
\end{equation*}
The choices made here are motivated by the main part of the operator $\Delta^{-1}DF(\overline{\textbf{U}})$ that we want to invert. Morally, we invert by hand the matrix $\begin{pmatrix} \gamma(\bar{v}) & \gamma'(\bar{v})\bar{u} \\ 0 & d \end{pmatrix}$ in $\ell^1_\nu$ in order to get the $w^{ij}$.

For later use, we also introduce
\begin{equation*}
    \quad \textbf{W}=
\left(\begin{array}{c|c}
    \textbf{w}^{11} & \textbf{w}^{12}  \\ \hline
    \textbf{w}^{21} & \textbf{w}^{22}
\end{array} \right).
\end{equation*}

\begin{rem} \label{rem:size}
    In all the following computations, we have $\bar{\textbf{u}}, \bar{\textbf{v}}$  belonging to $\ell^1_{\nu,N}$, with $N\in\mathbb{N}$. We want to keep the total information of the convolution of $\bar{\textbf{u}}$ with itself or with $\vecapp{\gamma(\bar{\textbf{v}})}$. Since we fix  $\vecapp{\gamma(\bar{\textbf{v}})} \in \ell^1_{\nu,N}$, we choose $\vecapp{F(\overline{\textbf{U}})} \in \ell^1_{\nu, 2N} \times \ell^1_{\nu, 2N}$. Then, to match, we take $\textbf{w}^{ij} \in \ell^1_{\nu,2N}$ and $A^{ij}_{11} \in \mathcal{L}(\ell^1_{\nu, 2N})$.
\end{rem}

\subsection{\texorpdfstring{The bound $Y$}{}}

The bound $Y$ is derived from the calculation of $\normenu{AF(\overline{\textbf{U}})}$. Let us do the math.
\begin{align*}
     AF(\overline{\textbf{U}})&= A \begin{pmatrix} \Delta(\gamma(\bar{\textbf{v}})*\bar{\textbf{u}})+\sigma\bar{\textbf{u}}*(\textbf{1}-\bar{\textbf{u}}) \\ d\Delta\bar{\textbf{v}}+\bar{\textbf{u}}-\bar{\textbf{v}}  \end{pmatrix}  \\
     &= A \begin{pmatrix}\Delta((\vecapp{\gamma(\bar{\textbf{v}})}+\vecdiff{\gamma(\bar{\textbf{v}})}) *\bar{\textbf{u}})+\sigma\bar{\textbf{u}}*(\textbf{1}-\bar{\textbf{u}}) \\ d\Delta\bar{\textbf{v}}+\bar{\textbf{u}}-\bar{\textbf{v}}\end{pmatrix} \\
     &= \underset{:=A\vecapp{F(\overline{\textbf{U}})}, \text{ completely known and computable exactly }}{\underbrace{A
     \begin{pmatrix}\Delta(\vecapp{\gamma(\bar{\textbf{v}})} *\bar{\textbf{u}})+\sigma\bar{\textbf{u}}*(\textbf{1}-\bar{\textbf{u}}) \\ d\Delta\bar{\textbf{v}}+\bar{\textbf{u}}-\bar{\textbf{v}}\end{pmatrix}}}
     + A\begin{pmatrix}\Delta(\vecdiff{\gamma(\bar{\textbf{v}})}*\bar{\textbf{u}}) \\ \textbf{0} \end{pmatrix}.
\end{align*}
We then focus on the rest,
\begin{align*}
    A\begin{pmatrix}\Delta(\vecdiff{\gamma(\bar{\textbf{v}})}*\bar{\textbf{u}}) \\ \textbf{0}\end{pmatrix} &= \left(\begin{array}{c|c}
    A^{11} & A^{12}  \\ \hline
    A^{21} & A^{22}
    \end{array} \right) \begin{pmatrix}\Delta(\vecdiff{\gamma(\bar{\textbf{v}})}*\bar{\textbf{u}}) \\ \textbf{0} \end{pmatrix} \\
    &= \begin{pmatrix} A^{11} \Delta(\vecdiff{\gamma(\bar{\textbf{v}})}*\bar{\textbf{u}}) \\
A^{21} \Delta(\vecdiff{\gamma(\bar{\textbf{v}})}*\bar{\textbf{u}})
\end{pmatrix} . 
\end{align*}
So,
\begin{align*}
\normenu{A\begin{pmatrix}\Delta(\vecdiff{\gamma(\bar{\textbf{v}})}*\bar{\textbf{u}}) \\ \textbf{0} \end{pmatrix}} &= \normenu{A^{11} \Delta(\vecdiff{\gamma(\bar{\textbf{v}})}*\bar{\textbf{u}})} + \normenu{A^{21}\Delta(\vecdiff{\gamma(\bar{\textbf{v}})}*\bar{\textbf{u}})} \\
    &\leq (\normenu{A^{11}\Delta}+\normenu{A^{21}\Delta}) \normenu{\vecdiff{\gamma(\bar{\textbf{v}})}*\bar{\textbf{u}}}  \\
    &\leq (\normenu{A^{11}\Delta}+\normenu{A^{21}\Delta}) \normenu{\vecdiff{\gamma(\bar{\textbf{v}})}}\normenu{\bar{\textbf{u}}} \\
    &\leq (\normenu{A^{11}\Delta}+\normenu{A^{21}\Delta}) \normenu{\bar{\textbf{u}}} \vecbound{\gamma(\bar{\textbf{v}})} .
\end{align*}
Thus we take,
\begin{equation*}
\boxed{
    Y = \normenu{A\vecapp{F(\overline{\textbf{U}})}} + (\normenu{A^{11}\Delta}+\normenu{A^{21}\Delta})  \normenu{\bar{\textbf{u}}} \vecbound{\gamma(\bar{\textbf{v}})}.
    }
\end{equation*}

It satisfies \eqref{boundY} in Theorem~\ref{th:NK}, $\normenu{AF(\overline{\textbf{U}})} \leq Y$.

\begin{rem}
    The computations made with the approximate vectors are said to be computable exactly. This is because they are finite (belong to a certain $\ell^1_{\nu,N}$ with $N$ not too large) and thanks to interval arithmetic \cite{intlabRump} these vectors can be calculated exactly.
\end{rem}

\subsection{\texorpdfstring{The bound $Z_1$}{}}
The bound $Z_1$ is derived to control $\normenu{I-ADF(\overline{\textbf{U}})}$. We use the same idea as before. We divide the terms related to $\gamma$ between parts we know exactly, and small remainders whose norm we can bound. Furthermore, we have to deal with an infinite dimensional operator, so we can only use the computer for a finite part, and then have to estimate the  ``tail'' by hand. Firstly, as above, we obtain
\begin{align*}
\normenu{I - ADF(\overline{\textbf{U}})} &\leq \normenu{I - A\vecapp{DF(\overline{\textbf{U}})}} + (\normenu{A^{11}\Delta}+\normenu{A^{21}\Delta})(\vecbound{\gamma(\bar{\textbf{v}})}+ \normenu{\bar{\textbf{u}}}\vecbound{\gamma'(\bar{\textbf{v}})}), \\
\text{ with } \vecapp{DF(\overline{\textbf{U}})} &= \begin{pmatrix} \Delta\vecapp{\gamma(\bar{\textbf{v}})}+\sigma(\textbf{1}-2\bar{\textbf{u}}) & \Delta(\vecapp{\gamma'(\bar{\textbf{v}})}*\bar{\textbf{u}}) \\ \textbf{1} & d\Delta-\textbf{1}\end{pmatrix}.
\end{align*}
We cannot compute exactly $B = I - A\vecapp{DF(\overline{\textbf{U}})}$ because $A$ and $\vecapp{DF(\overline{\textbf{U}})}$ are built with an infinite ``tail''. However, we can compute exactly a finite number of $B_{i,j}$. It leads us to separate between the columns whose norm we will compute numerically and exactly, and those we will bound by hand. With $K \geq 2N-1$ to be fixed later: 
\begin{align*}
    \normenu{B} &= \sup_{k\geq 0} \dfrac{1}{\xi_k(\nu)} \max \Big( \normenu{B^{11}_{(\cdot,k)}}+\normenu{B^{21}_{(\cdot,k)}}, \, \normenu{B^{12}_{(\cdot,k)}}+\normenu{B^{22}_{(\cdot,k)}} \Big) \\
    &= \max \Bigg(\sup_{k\leq K} \dfrac{1}{\xi_k(\nu)} \max \Big( \normenu{B^{11}_{(\cdot,k)}}+\normenu{B^{21}_{(\cdot,k)}}, \, \normenu{B^{12}_{(\cdot,k)}}+\normenu{B^{22}_{(\cdot,k)}} \Big), \\
    & \hphantom{\max\Bigg(\Bigg)} \sup_{k\geq K+1} \dfrac{1}{\xi_k(\nu)} \max \Big( \normenu{B^{11}_{(\cdot,k)}}+\normenu{B^{21}_{(\cdot,k)}}, \, \normenu{B^{12}_{(\cdot,k)}}+\normenu{B^{22}_{(\cdot,k)}} \Big) \ \Bigg).
\end{align*}

Let us denote, the ``finite'' part to compute,
\begin{align*}
    Z_{1,finite} &:= \displaystyle\sup_{k\leq K} \dfrac{1}{\xi_k(\nu)} \max \left( \normenu{B^{11}_{(\cdot,k)}}+\normenu{B^{21}_{(\cdot,k)}}, \, \normenu{B^{12}_{(\cdot,k)}}+\normenu{B^{22}_{(\cdot,k)}} \right),
\end{align*}
and the ``tail'' to bound,
\begin{align}
\mathcal{T}_{K,\nu}(B) &:= \sup_{k\geq K+1} \dfrac{1}{\xi_k(\nu)} \max \left( \normenu{B^{11}_{(\cdot,k)}}+\normenu{B^{21}_{(\cdot,k)}}, \, \normenu{B^{12}_{(\cdot,k)}}+\normenu{B^{22}_{(\cdot,k)}} \right) \notag \\ 
&= \underset{(\vectrunc[\textbf{z}]{K},\vectrunc[\textbf{z}']{K})=(\textbf{0},\textbf{0})}{\sup_{(\textbf{z},\textbf{z}') \neq (\textbf{0},\textbf{0})}} \dfrac{\normenu{B^{11}\textbf{z}+B^{12}\textbf{z}'} + \normenu{B^{21}\textbf{z}+B^{22}\textbf{z}'}}{\normenu{\textbf{z}}+\normenu{\textbf{z}'}} . \label{eq:tail}
\end{align}

\subsubsection{The finite part to compute}
Here we give an explanation of the calculation of $Z_{1,finite}$ made by the computer. 
To compute $Z_{1,finite}$, we have to understand correctly the calculation that determines $B^{ij}_{(\cdot,k)}$. 

We have for any $k\leq K$ and $(i,j)\in \lbrace 1,2\rbrace$, $B^{ij}_{(\cdot,k)} = \delta_{ij}I_{(\cdot,k)} - (A^{i1}\vecapp{DF(\overline{\textbf{U}})}^{1j}_{(\cdot,k)} + A^{i2}\vecapp{DF(\overline{\textbf{U}})}^{2j}_{(\cdot,k)}). $

Let $ l \in\lbrace 1,2\rbrace$, let us take a look at the column of $\vecapp{DF(\overline{\textbf{U}})}^{lj}$ (see Figure~\ref{fig:matrix} for a global view),
\begin{equation*}
    \vecapp{DF(\overline{\textbf{U}})}^{lj}_{(\cdot,k)} = \left \lbrace \begin{array}{ll} ( \underset{2N+k}{\underbrace{* \cdots \cdots *}} \ 0 \cdots )^\mathrm{T}, \ \text{ if } k\in \lbrace 0, \dots, 2N-1 \rbrace, \\
    ( \underset{k-2N+1}{\underbrace{0 \cdots \cdots 0}} \ \underset{4N-1 }{\underbrace{*\cdots \cdots *}} \ 0 \cdots )^\mathrm{T}, \ \text{ if } k\in \lbrace 2N ,\dots, K \rbrace.
    \end{array} \right.
\end{equation*}

Thus, to be exact in the matrix-vector product $A^{il}\vecapp{DF(\overline{\textbf{U}})}^{lj}_{(\cdot,k)}$, it is sufficient to take the first $k+2N$ rows of $A^{il}$. These rows have the same structure as the columns $\vecapp{DF(\overline{\textbf{U}})}^{lj}_{(\cdot,k)}$, except the transpose operation. Let $p\in \mathbb{N}$, we have the following equality:

\begin{equation*}
    \left(A^{il}\vecapp{DF(\overline{\textbf{U}})}^{lj}_{(\cdot,k)}\right)_{p} = \left \lbrace \begin{array}{ll} 
    \displaystyle\sum_{q=0}^{k+2N-1}A^{il}_{(p,q)}\vecapp{DF(\overline{\textbf{U}})}^{lj}_{(q,k)}, \text{ if } p \leq k+4N-2, \\
    0, \text{ if } p \geq k+4N-1 , \end{array} \right.
\end{equation*}
because for all $p\geq k+4N-1$, $A^{il}_{(p,q)} = 0, \, \forall q \in  \lbrace0, \dots, k+2N-1\rbrace \subset \lbrace0,\dots, p-2N\rbrace$.

The situation is summarized on Figure~\ref{fig:matrix}. Remember we start counting at 0 for the columns and the rows.

\begin{center}
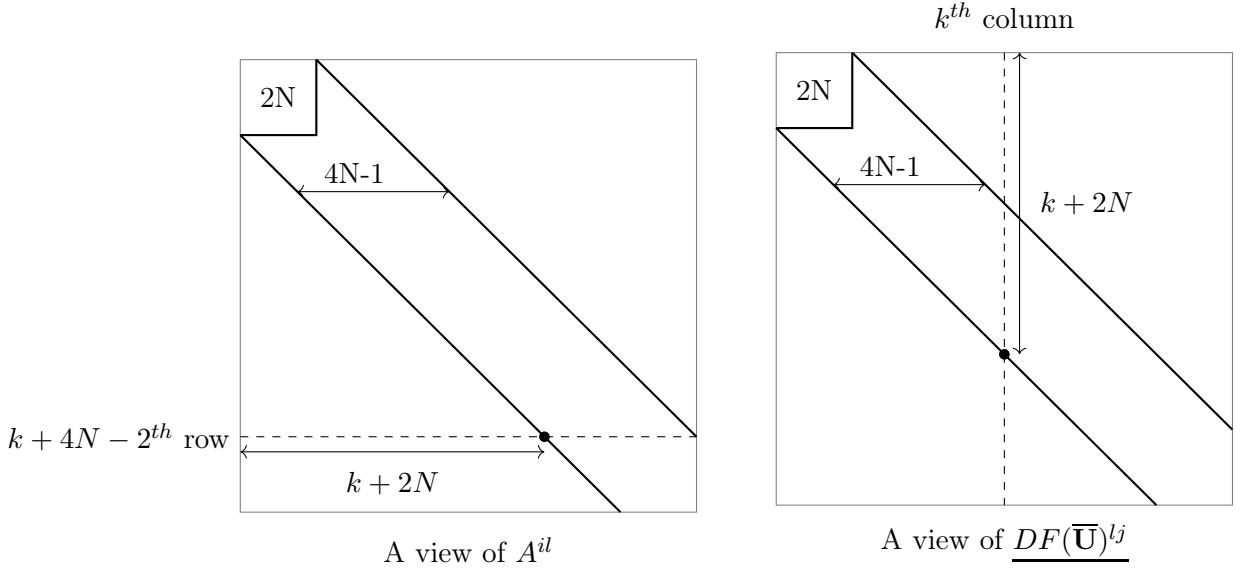
\begin{figure}[ht]
\begin{tikzpicture}
\draw[gray, thin] (0,1) rectangle (6,-5);
\draw[black, thick] (0,0) -- (1,0) -- (1,1);
\node at (0.5,0.5) {2N};
\draw[black, thick] (1,1) -- (6,-4);
\draw[black, thick] (0,0) -- (5,-5);
\node at (1.5,-0.5) {4N-1};
\draw[black, <->] (0.75, -0.75) -- (2.75,-0.75);
\node at (-1.6,-4) {$k + 4N-2^{th}$ row};
\draw[black, dashed] (0,-4) -- (6,-4);
\fill[black] (4,-4) circle (2pt);
\draw[black, <->] (0,-4.2) -- (4,-4.2) node[midway,label=below:$k + 2N$] {};
\node at (3,-5.5) {A view of $A^{il}$};
\end{tikzpicture}
\hspace{2em}
\begin{tikzpicture}
\draw[gray, thin] (0,1) rectangle (6,-5);
\draw[black, thick] (0,0) -- (1,0) -- (1,1);
\node at (0.5,0.5) {2N};
\draw[black, thick] (1,1) -- (6,-4);
\draw[black, thick] (0,0) -- (5,-5);
\node at (1.5,-0.5) {4N-1};
\draw[black, <->] (0.75, -0.75) -- (2.75,-0.75);
\node at (3,1.5) {$k^{th}$ column};
\draw[black, dashed] (3,1) -- (3,-5);
\fill[black] (3,-3) circle (2pt);
\draw[black, <->] (3.2,1) -- (3.2,-3) node[midway,label=right:$k + 2N$] {};
\node at (3,-5.5) {A view of $\vecapp{DF(\overline{\textbf{U}})^{lj}}$};
\end{tikzpicture}
\caption{Figure of the matrices}
\label{fig:matrix}
\end{figure}
\end{center}

Then, because we are interested in the first $K$ columns of $B^{ij}$, this is sufficient to take the first $K$ columns and the first $K+2N$ rows of $\vecapp{DF(\overline{\textbf{U}})^{lj}}, \ l \in\lbrace 1,2\rbrace$ and the first $K+2N$ columns and the first $K+4N-1$ rows of $A^{il}, \ l \in\lbrace 1,2\rbrace$ to compute exactly $Z_{1,finite}$.

\subsubsection{The infinite part to bound}
Let us bound the ``tail'' $\mathcal{T}_{K,\nu}(B)$ \eqref{eq:tail}. Let $K = 2N-1$, let $\textbf{z}\in \ell^1_\nu$ with $\vectrunc[\textbf{z}]{K}= 0$ and $\textbf{z}\neq 0$. This choice of $K$ means we look for the columns just after the finite part, this is sufficient to obtain the following calculations on $Z_{1,tail}$ and so $Z_1$.

Because $K \geq 2N-1$, we simply get
\begin{align*}
    B^{ij}\textbf{z}= \delta_{ij}\textbf{z}- \left( \mathrm{M}(\textbf{w}^{i1})\Delta^{-1} \vecapp{DF(\overline{\textbf{U}})}^{1j} + \mathrm{M}(\textbf{w}^{i2})\Delta^{-1}\vecapp{DF(\overline{\textbf{U}})}^{2j}\right)\textbf{z}.
\end{align*}
Indeed, the top-right block of $A^{ij}$, namely $A^{ij}_{11}$ only acts on the $2N$ (from $0$ to $2N-1$) first coefficients of $\textbf{z}$ which are null.

Recall that $\vecapp{DF(\overline{\textbf{U}})}^{ij} = \Delta \vecapp{D\Phi(\overline{\textbf{U}})}^{ij} + \vecapp{DR(\overline{\textbf{U}})}^{ij}$,
with 
\begin{align*}
    \overline{\textbf{U}} = \begin{pmatrix}\bar{\textbf{u}} \\ \bar{\textbf{v}} \end{pmatrix},\quad \Phi\begin{pmatrix}\bar{\textbf{u}} \\ \bar{\textbf{v}} \end{pmatrix} = \begin{pmatrix}\gamma(\bar{\textbf{v}}) * \bar{\textbf{u}} \\ d\bar{\textbf{v}} \end{pmatrix} \quad \text{and}\quad R\begin{pmatrix}\bar{\textbf{u}} \\ \bar{\textbf{v}} \end{pmatrix} =\begin{pmatrix}\sigma \bar{\textbf{u}}*( \textbf{1} -\bar{\textbf{u}}) \\ \bar{\textbf{u}}-\bar{\textbf{v}}\end{pmatrix}.
\end{align*}

Let us  compute $B^{ij}\textbf{z}$,
\begin{align*}
     B^{ij}\textbf{z} =\, & \delta_{ij}\textbf{z} - \left( \mathrm{M}(\textbf{w}^{i1}) \vecapp{D\Phi(\overline{\textbf{U}})}^{1j} + \mathrm{M}(\textbf{w}^{i2})\vecapp{D\Phi(\overline{\textbf{U}})}^{2j}\right)\textbf{z} \\
     &- \left(\mathrm{M}(\textbf{w}^{i1})\Delta^{-1} \vecapp{DR(\overline{\textbf{U}})}^{1j} + \mathrm{M}(\textbf{w}^{i2})\Delta^{-1}\vecapp{DR(\overline{\textbf{U}})}^{2j}\right)\textbf{z}.
\end{align*}
Since we have, by associativity, $\mathrm{M}(\textbf{w}^{ij})\vecapp{D\Phi(\overline{\textbf{U}})}^{jk}\textbf{z} = \mathrm{M}(\textbf{w}^{ij})(\vecapp{D\Phi(\overline{\textbf{U}})}^{jk}*\textbf{z}) = \textbf{w}^{ij}*(\vecapp{D\Phi(\overline{\textbf{U}})}^{jk}*\textbf{z}) = (\textbf{w}^{ij}*\vecapp{D\Phi(\overline{\textbf{U}})}^{jk})*\textbf{z}$.

Then,
\begin{align*}
B^{ij}\textbf{z} &= \left[ \left(  \begin{array}{c|c}
     \textbf{1} & \textbf{0}  \\ \hline
     \textbf{0} & \textbf{1} 
\end{array} \right) - \textbf{W}*\vecapp{D\Phi(\overline{\textbf{U}})} \right]^{ij}*\textbf{z} - \textbf{w}^{ij}*\left[ \Delta^{-1}\left(\vecapp{DR(\overline{\textbf{U}})}^{ij}*\textbf{z}\right)\right].
\end{align*}
Finally, with $(\textbf{z},\textbf{z}')$ such that $ (\vectrunc[\textbf{z}]{K},\vectrunc[\textbf{z}']{K}) = (\textbf{0},\textbf{0})$,
\begin{align*}
    B\begin{pmatrix}\textbf{z} \\ \textbf{z}' \end{pmatrix} &= \left( \left( 
    \begin{array}{c|c}
    \textbf{1} & \textbf{0}  \\ \hline
    \textbf{0} & \textbf{1} 
    \end{array} \right) - \textbf{W}*\vecapp{D\Phi(\overline{\textbf{U}})} \right)*\begin{pmatrix}\textbf{z} \\ \textbf{z}' \end{pmatrix} - \textbf{W}*\left[\left(\begin{array}{c|c}
     \Delta^{-1} & \textbf{0}  \\ \hline
     \textbf{0} & \Delta^{-1} 
\end{array} \right)\left(\vecapp{DR(\overline{\textbf{U}})}*\begin{pmatrix}\textbf{z}\\ \textbf{z}' \end{pmatrix}\right)\right] .
\end{align*}
So,
\begin{align*}
\normenu{B\begin{pmatrix}\textbf{z} \\ \textbf{z}'\end{pmatrix}} &\leq 
\normenu{\left( \left(  
    \begin{array}{c|c}
        \textbf{1} & \textbf{0}  \\ \hline
        \textbf{0} & \textbf{1} 
    \end{array} \right)  - \textbf{W} * \vecapp{D\Phi(\overline{\textbf{U}})}\right) \begin{pmatrix}\textbf{z} \\ \textbf{z}' \end{pmatrix} } + \normenu{\textbf{W}*\left[\left(\begin{array}{c|c}
     \Delta^{-1} & \textbf{0}  \\ \hline
     \textbf{0} & \Delta^{-1} 
\end{array} \right)\left(\vecapp{DR(\overline{\textbf{U}})}*\begin{pmatrix}\textbf{z}\\ \textbf{z}' \end{pmatrix}\right)\right]} \\
    &\leq \normenu{\left( \left(  
    \begin{array}{c|c}
        \textbf{1} & \textbf{0}  \\ \hline
        \textbf{0} & \textbf{1} 
    \end{array} \right)  - \textbf{W} * \vecapp{D\Phi(\overline{\textbf{U}})} \right)\begin{pmatrix}\textbf{z} \\ \textbf{z}' \end{pmatrix} } + \normenu{\textbf{W}} \normenu{\left(\begin{array}{c|c}
     \Delta^{-1} & \textbf{0}  \\ \hline
     \textbf{0} & \Delta^{-1} 
\end{array} \right)\left(\vecapp{DR(\overline{\textbf{U}})}*\begin{pmatrix}\textbf{z}\\ \textbf{z}' \end{pmatrix}\right) } \, .
\end{align*}

Also,
\begin{align*}
    \normenu{\left(\begin{array}{c|c}
     \Delta^{-1} & \textbf{0}  \\ \hline
     \textbf{0} & \Delta^{-1} 
\end{array} \right)\left(\vecapp{DR(\overline{\textbf{U}})}*\begin{pmatrix}\textbf{z}\\ \textbf{z}' \end{pmatrix}\right) } &= \normenu{\left(\begin{array}{c|c}
     \Delta^{-1} & \textbf{0}  \\ \hline
     \textbf{0} & \Delta^{-1} 
\end{array} \right)\begin{pmatrix}\sigma(\textbf{1}-2\bar{\textbf{u}})*\textbf{z}\\ \textbf{z}-\textbf{z}' \end{pmatrix}} \\
&= \normenu{\begin{pmatrix} \sigma\Delta^{-1}\left[(\textbf{1}-2\bar{\textbf{u}})*\textbf{z}\right]\\ \Delta^{-1}\left[\textbf{z}-\textbf{z}'\right] \end{pmatrix}} \\
&= \sigma \normenu{\Delta^{-1}\left[(\textbf{1}-2\bar{\textbf{u}})*\textbf{z}\right]} + \normenu{\Delta^{-1}\left[\textbf{z}-\textbf{z}'\right]} \\
&\leq \sigma \dfrac{(b-a)^2}{\pi^2(K-N+1)^2}\normenu{(\textbf{1}-2\bar{\textbf{u}})*\textbf{z}} + \dfrac{(b-a)^2}{\pi^2(K+1)^2}\normenu{\textbf{z}-\textbf{z}'} \\
&\leq \dfrac{(b-a)^2}{\pi^2(K-N+1)^2} \left(\sigma\normenu{(\textbf{1}-2\bar{\textbf{u}})*\textbf{z}} + \normenu{\textbf{z}}+ \normenu{\textbf{z}'} \right).
\end{align*}

\begin{rem}
    Since $\bar{\textbf{u}} \in \ell^1_{\nu,2N}$ and $\textbf{z} \in \ell^1_\nu, \ \vectrunc[\textbf{z}]{K} = \textbf{0}$, we have $\vectrunc[\bar{\textbf{u}}*\textbf{z}]{K-N}= \textbf{0}$.
    Furthermore, for any $\textbf{x}\in \ell^1_\nu$ such that $\vectrunc[\textbf{x}]{K_\textbf{x}} = \textbf{0}$ we have $ \ \normenu{\Delta^{-1} \textbf{x}} \leq \dfrac{(b-a)^2}{\pi^2(K_{\textbf{x}}+1)^2}\normenu{\textbf{x}}$. This explains the appearance of the $\dfrac{(b-a)^2}{\pi^2(K-N+1)^2}$ term in the above estimate.
\end{rem}

Thus, thanks to the precedent calculations and equation~\eqref{eq:tail}, we have
\begin{align*}
    \mathcal{T}_{K,\nu}(B) &\leq \mathcal{T}_{K,\nu}\left( \left(  \begin{array}{c|c}\textbf{1} & \textbf{0}  \\ \hline \textbf{0} & \textbf{1} \end{array} \right)  - \textbf{W} * \vecapp{{D\Phi(\overline{\textbf{U}})}}\right) \\
    & \hphantom{===}+ \normenu{\textbf{W}} \dfrac{(b-a)^2}{\pi^2(K-N+1)^2} \underset{(\vectrunc[\textbf{z}]{K},\vectrunc[\textbf{z}']{K})=(\textbf{0},\textbf{0})}{\sup_{(\textbf{z},\textbf{z}')\neq(\textbf{0},\textbf{0})}} \left( \dfrac{\sigma\normenu{(\textbf{1}-2\bar{\textbf{u}})*\textbf{z}} + \normenu{\textbf{z}}+ \normenu{\textbf{z}'}}{\normenu{\textbf{z}}+\normenu{\textbf{z}'}} \right).
\end{align*}

Firstly,
\begin{align*}
\mathcal{T}_{K,\nu}\left( \left(  \begin{array}{c|c}\textbf{1} & \textbf{0}  \\ \hline \textbf{0} & \textbf{1} \end{array} \right)  - \textbf{W} * \vecapp{{D\Phi(\overline{\textbf{U}})}}\right) 
    &\leq \normenu{\left(  \begin{array}{c|c}\textbf{1} & \textbf{0}  \\ \hline \textbf{0} & \textbf{1} \end{array} \right)  - \textbf{W} * \vecapp{{D\Phi(\overline{\textbf{U}})}}}.  
\end{align*}
Secondly,
\begin{align*}
    \underset{(\vectrunc[\textbf{z}]{K},\vectrunc[\textbf{z}']{K})=(\textbf{0},\textbf{0})}{\sup_{(\textbf{z},\textbf{z}')\neq(\textbf{0},\textbf{0})}} \left( \dfrac{\sigma\normenu{(\textbf{1}-2\bar{\textbf{u}})*\textbf{z}} + \normenu{\textbf{z}}+ \normenu{\textbf{z}'}}{\normenu{\textbf{z}}+\normenu{\textbf{z}'}} \right) 
    &\leq \underset{(\vectrunc[\textbf{z}]{K},\vectrunc[\textbf{z}']{K})=(\textbf{0},\textbf{0})}{\sup_{(\textbf{z},\textbf{z}')\neq(\textbf{0},\textbf{0})}} \left( \dfrac{\sigma\normenu{(\textbf{1}-2\bar{\textbf{u}})*\textbf{z}}}{\normenu{\textbf{z}}+\normenu{\textbf{z}'}}\right) + 1  \\
    &\leq \underset{\vectrunc[\textbf{z}]{K}=\textbf{0}}{\sup_{\textbf{z}\neq\textbf{0}}} \left(\dfrac{\sigma\normenu{(\textbf{1}-2\bar{\textbf{u}})*\textbf{z}}}{\normenu{\textbf{z}}} \right) + 1 \\
    &\leq \sigma\normenu{\textbf{1}-2\bar{\textbf{u}}} + 1 \, . \\
\end{align*}
Then, we have
\begin{align*}
    \mathcal{T}_{K,\nu}(B) &\leq \normenu{\left(  \begin{array}{c|c}\textbf{1} & \textbf{0}  \\ \hline \textbf{0} & \textbf{1} \end{array} \right)  - \textbf{W} * \vecapp{{D\Phi(\overline{\textbf{U}})}}} + \dfrac{(b-a)^2}{\pi^2(K-N+1)^2}\normenu{\textbf{W}}\left(\sigma\normenu{\textbf{1}-2\bar{\textbf{u}}} + 1 \right) \\
    &:= Z_{1,tail}.
\end{align*}

All of this is \emph{computable exactly}, and we take 
\begin{equation*}
\boxed{Z_1 = Z_{1,finite}+Z_{1,tail}+(\normenu{A^{11}\Delta}+\normenu{A^{21}\Delta})(\vecbound{\gamma(\bar{\textbf{v}})}+ \normenu{\bar{\textbf{u}}}\vecbound{\gamma'(\bar{\textbf{v}})}).}
\end{equation*} 
It satisfies \eqref{boundZ1} in Theorem~\ref{th:NK}, $\normenu{I-ADF(\overline{\textbf{U}})} \leq Z_1$.

\begin{rem}
    Notice that the injectivity of $A$, which is part of the assumptions needed in Theorem~\ref{th:NK} in order to obtain the existence of a zero of $F$, is in fact a consequence of the other assumptions in this setting. Indeed, as soon as assumption~\eqref{condNKa} is satisfied, then $Z_1<1$, which implies that $A$ is injective (see e.g. the proof of~\cite[Theorem 2.5]{B22}). 
\end{rem}

\subsection{\texorpdfstring{The bound $Z_2$}{}}
We assume here that we have done the work of Section~\ref{sec:tool} on $\gamma$ and its derivatives.
To compute $Z_2$ we need to compute the second Fréchet derivative of $F$. Let $\textbf{U} = \begin{pmatrix}\textbf{u} \\ \textbf{v} \end{pmatrix} \in (\ell^1_\nu)^2$. 

Let $\begin{pmatrix}\textbf{u}_1 \\ \textbf{v}_1 \end{pmatrix}$ and $\begin{pmatrix}\textbf{u}_2 \\ \textbf{v}_2 \end{pmatrix}$ belong to $(\ell^1_\nu)^2$,

\begin{align*}
    D^2F\begin{pmatrix}\textbf{u} \\ \textbf{v} \end{pmatrix}\left( \begin{pmatrix}\textbf{u}_1 \\ \textbf{v}_1 \end{pmatrix},\begin{pmatrix}\textbf{u}_2 \\ \textbf{v}_2 \end{pmatrix} \right) &= \Delta D^2\Phi\left( \begin{pmatrix}\textbf{u}_1 \\ \textbf{v}_1 \end{pmatrix},\begin{pmatrix}\textbf{u}_2 \\ \textbf{v}_2 \end{pmatrix} \right) + D^2R\left( \begin{pmatrix}\textbf{u}_1 \\ \textbf{v}_1 \end{pmatrix},\begin{pmatrix}\textbf{u}_2 \\ \textbf{v}_2 \end{pmatrix} \right),
\end{align*}
where 
\begin{align*}
    D^2\Phi\begin{pmatrix}\textbf{u} \\ \textbf{v} \end{pmatrix}\left( \begin{pmatrix}\textbf{u}_1 \\ \textbf{v}_1 \end{pmatrix},\begin{pmatrix}\textbf{u}_2 \\ \textbf{v}_2 \end{pmatrix} \right) &= \begin{pmatrix}\gamma'(\textbf{v})*\textbf{u}_1*\textbf{v}_2+\gamma'(\textbf{v})*\textbf{v}_1*\textbf{u}_2+\gamma''(\textbf{v})*\textbf{u}*\textbf{v}_1*\textbf{v}_2 \\ 0 \end{pmatrix},
\end{align*}
and 
\begin{align*}
    D^2R\begin{pmatrix}\textbf{u} \\ \textbf{v} \end{pmatrix}\left( \begin{pmatrix}\textbf{u}_1 \\ \textbf{v}_1 \end{pmatrix},\begin{pmatrix}\textbf{u}_2 \\ \textbf{v}_2 \end{pmatrix} \right) &= \begin{pmatrix}-2\sigma\textbf{u}*\textbf{u}_1*\textbf{u}_2 \\ 0 \end{pmatrix}.
\end{align*}
Let us denote
\begin{align*}
    S  \left(\begin{pmatrix}\textbf{u}_1 \\ \textbf{v}_1 \end{pmatrix}, \begin{pmatrix}\textbf{u}_2 \\ \textbf{v}_2 \end{pmatrix}\right) &= \Delta\left(\gamma'(\textbf{v})*\textbf{u}_1*\textbf{v}_2+\gamma'(\textbf{v})*\textbf{v}_1*\textbf{u}_2+\gamma''(\textbf{v})*\textbf{u}*\textbf{v}_1*\textbf{v}_2\right)  -2\sigma\textbf{u}*\textbf{u}_1*\textbf{u}_2.
\end{align*}
We have 
\begin{align*}
    AD^2F\begin{pmatrix}\textbf{u} \\ \textbf{v} \end{pmatrix}\left( \begin{pmatrix}\textbf{u}_1 \\ \textbf{v}_1 \end{pmatrix},\begin{pmatrix}\textbf{u}_2 \\ \textbf{v}_2 \end{pmatrix} \right) 
    &= \left(\begin{array}{c|c}
     A^{11} & A^{12}  \\ \hline
     A^{21} & A^{22}
    \end{array} \right) D^2F\begin{pmatrix}\textbf{u} \\ \textbf{v} \end{pmatrix}\left( \begin{pmatrix}\textbf{u}_1 \\ \textbf{v}_1 \end{pmatrix},\begin{pmatrix}\textbf{u}_2 \\ \textbf{v}_2 \end{pmatrix} \right) \\
    &= \begin{pmatrix}A^{11}S \\ A^{21}S \end{pmatrix}.
\end{align*}
So
\begin{align*}
    \normenu{AD^2F\begin{pmatrix}\textbf{u} \\ \textbf{v} \end{pmatrix}\left( \begin{pmatrix}\textbf{u}_1 \\ \textbf{v}_1 \end{pmatrix},\begin{pmatrix}\textbf{u}_2 \\ \textbf{v}_2 \end{pmatrix} \right)} &= \normenu{A^{11}S\left( \begin{pmatrix}\textbf{u}_1 \\ \textbf{v}_1 \end{pmatrix},\begin{pmatrix}\textbf{u}_2 \\ \textbf{v}_2 \end{pmatrix}\right)}+\normenu{A^{21}S\left( \begin{pmatrix}\textbf{u}_1 \\ \textbf{v}_1 \end{pmatrix},\begin{pmatrix}\textbf{u}_2 \\ \textbf{v}_2 \end{pmatrix}\right)}.
\end{align*}
We estimate 
\begin{align*}
    \normenu{A^{ij}S} = \sup_{\left( \begin{pmatrix}\textbf{u}_1 \\ \textbf{v}_1 \end{pmatrix}, \begin{pmatrix} \textbf{u}_2 \\ \textbf{v}_2 \end{pmatrix} \right)\neq (\textbf{0},\textbf{0})} \dfrac{\normenu{A^{ij} S \left( \begin{pmatrix}\textbf{u}_1 \\ \textbf{v}_1 \end{pmatrix},\begin{pmatrix}\textbf{u}_2 \\ \textbf{v}_2 \end{pmatrix} \right)}}{\normenu{\begin{pmatrix} \textbf{u}_1\\ \textbf{v}_1 \end{pmatrix}} \cdot \ \normenu{\begin{pmatrix}\textbf{u}_2 \\ \textbf{v}_2 \end{pmatrix}}} .
\end{align*}
\begin{align*}
    \normenu{A^{ij} S \left( \begin{pmatrix}\textbf{u}_1 \\ \textbf{v}_1 \end{pmatrix},\begin{pmatrix}\textbf{u}_2 \\ \textbf{v}_2 \end{pmatrix} \right)} &= 
    \normenu{A^{ij}\left( \Delta(\gamma'(\textbf{v})*(\textbf{u}_1*\textbf{v}_2+\textbf{v}_1*\textbf{u}_2)+\gamma''(\textbf{v})*\textbf{u}*\textbf{v}_1*\textbf{v}_2)  -2\sigma \textbf{u}*\textbf{u}_1*\textbf{u}_2 \right)} \\
    &\leq \normenu{A^{ij}\Delta(\gamma'(\textbf{v})*(\textbf{u}_1*\textbf{v}_2+\textbf{v}_1*\textbf{u}_2)) } + \normenu{A^{ij}\Delta(\gamma''(\textbf{v})*\textbf{u}*\textbf{v}_1*\textbf{v}_2)} \\
    &\hphantom{=}+ 2\sigma \normenu{A^{ij}\textbf{u}*\textbf{u}_1*\textbf{u}_2} \\
    &\leq \normenu{A^{ij}\Delta}\normenu{\gamma'(\textbf{v})}(\normenu{\textbf{u}_1}\normenu{\textbf{v}_2}+\normenu{\textbf{v}_1}\normenu{\textbf{u}_2}) \\
    &\hphantom{=}+ \normenu{A^{ij}\Delta}\normenu{\gamma''(\textbf{v})*\textbf{u}}\normenu{\textbf{v}_1}\normenu{\textbf{v}_2} \\
    &\hphantom{=}+ 2\sigma \normenu{A^{ij}\textbf{u}}\normenu{\textbf{u}_1}\normenu{\textbf{u}_2}.
\end{align*}
So 
\begin{align*}
    \normenu{AD^2F\begin{pmatrix}\textbf{u} \\ \textbf{v} \end{pmatrix}\left( \begin{pmatrix}\textbf{u}_1 \\ \textbf{v}_1 \end{pmatrix},\begin{pmatrix}\textbf{u}_2 \\ \textbf{v}_2 \end{pmatrix} \right)} &\leq (\normenu{A^{11}\Delta}+\normenu{A^{21}\Delta}) \normenu{\gamma'(\textbf{v})}(\normenu{\textbf{u}_1}\normenu{\textbf{v}_2}+\normenu{\textbf{v}_1}\normenu{\textbf{u}_2}) \\
    &\hphantom{=}+  (\normenu{A^{11}\Delta}+\normenu{A^{21}\Delta})\normenu{\gamma''(\textbf{v})*\textbf{u}} \normenu{\textbf{v}_1}\normenu{\textbf{v}_2} \\
    &\hphantom{=}+ 2\sigma (\normenu{A^{11}\textbf{u}} + \normenu{A^{21}\textbf{u}})\normenu{\textbf{u}_1}\normenu{\textbf{u}_2} .
\end{align*}
Finally, 
\begin{align} \label{eq:AD2F}
    \normenu{AD^2F\begin{pmatrix}\textbf{u} \\ \textbf{v} \end{pmatrix}} 
    \leq \max \Big( &(\normenu{A^{11}\Delta}+\normenu{A^{21}\Delta})\normenu{\gamma'(\textbf{v})}, \notag \\
    &(\normenu{A^{11}\Delta}+\normenu{A^{21}\Delta})\normenu{\gamma''(\textbf{v})*\textbf{u}}, \notag \\
    &2\sigma (\normenu{A^{11}\textbf{u}} + \normenu{A^{21}\textbf{u}}) \Big)\,.
\end{align}

In order to use Theorem~\ref{th:NK}, we have to get such an estimate for all $\begin{pmatrix} \textbf{u} \\ \textbf{v} \end{pmatrix}$ the neighbourhood of $\begin{pmatrix}\bar{\textbf{u}} \\ \bar{\textbf{v}} \end{pmatrix}$.

Let $r^*>0$ and $\begin{pmatrix} \textbf{u} \\ \textbf{v} \end{pmatrix} \in \mathcal{B}_\nu \left(\begin{pmatrix} \bar{\textbf{u}} \\ \bar{\textbf{v}} \end{pmatrix},r^*\right).$ We have directly $\normenu{\textbf{u}-\bar{\textbf{u}}} \leq r^*$ and $\normenu{\textbf{v}-\bar{\textbf{v}}} \leq r^*$. We can therefore apply the results of Section~\ref{sec:tool} with $\vecapp{\textbf{v}}=\bar{\textbf{v}}$ and $\vecbound{\textbf{v}}=r^*$, and $\vecapp{\textbf{u}}=\bar{\textbf{u}}$ and $\vecbound{\textbf{u}}=r^*$. In particular, we get a the full description of $\gamma'(\textbf{v})$ with $\vecapp{\gamma'(\bar{\textbf{v}})}$ and $\vecbound{\gamma'(\textbf{v})}$ that depends on $r^*$, and similarly for $\gamma''(\textbf{v})$. Thus we have,
\begin{align*}
    \textbf{u} &= \bar{\textbf{u}} + \vecdiff{\textbf{u}}, \text{ with }
    \normenu{\vecdiff{\textbf{u}}} \leq r^*, \\
    \gamma'(\textbf{v}) &= \vecapp{\gamma'(\bar{\textbf{v}})} + \vecdiff{\gamma'(\textbf{v})}, \text{ with }
    \normenu{\vecdiff{\gamma'(\textbf{v})}} \leq \vecbound{\gamma'(\textbf{v})}, \\
    \gamma''(\textbf{v}) &= \vecapp{\gamma''(\bar{\textbf{v}})} + \vecdiff{\gamma''(\textbf{v})}, \text{ with }
    \normenu{\vecdiff{\gamma''(\textbf{v})}} \leq \vecbound{\gamma''(\textbf{v})}.
\end{align*}

We inject these identities into the inequality~\eqref{eq:AD2F} and use triangle inequality to obtain $Z_2$.

To clarify, let us define:
\begin{align*}
    Z_2^a &= (\normenu{A^{11}\Delta} + \normenu{A^{21}\Delta} ) (\normenu{\vecapp{\gamma'(\textbf{v})}}+ \vecbound{\gamma'(\bar{\textbf{v}})}) , \\
    Z_2^b &= (\normenu{A^{11}\Delta} + \normenu{A^{21}\Delta})\Big[\normenu{\vecapp{\gamma''(\bar{\textbf{v}})}*\bar{\textbf{u}}} +\normenu{\vecapp{\gamma''(\bar{\textbf{v}})}}r^* + \left(\normenu{\bar{\textbf{u}}}+r^*\right)\vecbound{\gamma''(\textbf{v})}\Big], \\
    Z_2^c &= 2\sigma [\normenu{A^{11}\bar{\textbf{u}}} + \normenu{A^{21}\bar{\textbf{u}}}  + (\normenu{A^{11}} + \normenu{A^{21}})r^*] .
\end{align*}
Finally we have
\begin{equation*}
    \boxed{Z_2 = \max(Z_2^a,Z_2^b,Z_2^c) .}
\end{equation*}
It satisfies \eqref{boundZ2} in Theorem~\ref{th:NK}, $\normenu{AD^2F(\textbf{U})} \leq Z_2, \ \forall \textbf{U} \in \mathcal{B}_\nu(\overline{\textbf{U}},r^*)$.

\section{Proofs of the existence theorems}
\label{sec:proofs}

In this last section, we specify how to apply the techniques presented in this paper to prove Theorem~\ref{th:single_sol}, Theorem~\ref{th:WX} and Theorem~\ref{th:bif_diag}. All the computer-assisted parts of the proofs can be reproduced using the code available at~\cite{GitHub}.

\subsection{Proof of Theorem~\ref{th:single_sol}} \label{sec:proof_single}
Let $\sigma = 0.053$, $d = 1$, $\Omega=(0,3\pi)$ and $\gamma(x) = \dfrac{1}{1+x^9}$. The approximate steady state $(\bar{u},\bar{v})$ is depicted in Figure~\ref{fig:single}, the complete description of $\bar{u}$ and $\bar{v}$ in terms of Fourier coefficients can be found in the file \texttt{Datas/U\_initial\_single\_solution.mat} from~\cite{GitHub}. Given $r^* = 1\times 10^{-6}$, $N=100$ and $\nu = 1.0001$, we compute the bounds $Y,\ Z_1$ and $Z_2$ for this approximate solution according to Section~\ref{sec:bounds} and obtain the following results\footnote{The results shown are the upper bounds of each quantity calculated with \textsc{Matlab} and \textsc{Intlab}. We write the results in scientific notation to 5 significant digits.}:
\begin{itemize}
    \item[] $ Y = 2.4051 \times 10^{-8}, $
    \item[] $ Z_1 = 3.1194 \times 10^{-2},$
    \item[] $ Z_2 = 3.6100\times 10^{4} .$
\end{itemize}
We then check that condition~\eqref{condNK} is satisfied. Theorem~\ref{th:NK} therefore applies, and yields the existence of a locally unique zero $(\textbf{u},\textbf{v})$ of $F$ satisfying 
\begin{align*}
    \normenu{(\textbf{u},\textbf{v})-(\bar{\textbf{u}},\bar{\textbf{v}})} &\leq 2.5197\times 10^{-8}.
\end{align*}

The estimate announced in Theorem~\ref{th:single_sol} then simply follows from the fact that $ \sup_{x\in\Omega} \vert z(x)\vert \leq \normenu{\textbf{z}}$.

The computational parts of the proof (i.e. the evaluation of $Y,\ Z_1$ and $Z_2$ and the verification of assumption~\eqref{condNK}) can be reproduced by running the \textsc{Matlab} code \texttt{proof\_single\_solution.m} from~\cite{GitHub}, together with \textsc{Intlab}~\cite{intlabRump}.

\subsection{Proof of Theorem~\ref{th:WX}}
Let $\sigma = 0.6$, $d = 1$, $\Omega=(0,4\pi)$ and $\gamma(x) = \dfrac{1}{1+\e{9(x-1)}}$. The approximated steady state $(\bar{u},\bar{v})$ is described in Figure~\ref{fig:thWX}, the complete description of $\bar{u}$ and $\bar{v}$ in terms of Fourier coefficients can be found in the file \texttt{Datas/U\_initial\_WX.mat} after a Newton's method done in the proof. 
Given $r^* = 1\times 10^{-6}$, $N=100$ and $\nu = 1.0001$, for this approximate solution according to Section~\ref{sec:bounds} and obtain the following results\footnotemark[1]:
\begin{itemize}
    \item[] $ Y =  1.5327 \times 10^{-12},$
    \item[] $ Z_1 = 2.4338\times 10^{-2},$
    \item[] $ Z_2 = 6.4843\times 10^{2} .$
\end{itemize}
As before, the condition \eqref{condNK} is satisfied. Again thanks to Theorem~\ref{th:NK}, there exists $(\textbf{u},\textbf{v})$ a locally unique zero of $F$ satisfying
\begin{align*}
    \normenu{(\textbf{u},\textbf{v})-(\bar{\textbf{u}},\bar{\textbf{v}})} &\leq 1.6956\times 10^{-12}.
\end{align*}

The computational parts of the proof can be reproduced by running the \textsc{Matlab} code \texttt{proof\_WX.m} from~\cite{GitHub}, together with \textsc{Intlab}~\cite{intlabRump}.

\subsection{Proof of Theorem~\ref{th:bif_diag}}
Let $d = 1$, $\Omega=(0,3\pi)$ and $\gamma(x) = \dfrac{1}{1+x^9}$.
To prove the existence of solution for each point of the discrete diagram Figure~\ref{fig:th-discrete-diag}, we follow the exact same method as the proof of Section~\ref{sec:proof_single}. 
Let \texttt{Datas/UU\_discrete\_diagram.mat} be all the data of the each solution  shown in Figure~\ref{fig:th-discrete-diag}. Let $r^* = 1 \times 10^{-8}$, $N=150$ and $\nu = 1.0001$, for $i$ from $1$ to $\texttt{length(UU)}$ we compute the bounds $Y, \ Z_1$ and $Z_2$ associated to the system with the current approximate solution $(\bar{\textbf{u}},\bar{\textbf{v}}) = \texttt{UU(2:end,i)}$ with the $\sigma = \texttt{UU(1,i)}$ associated. We check if the condition \eqref{condNK} is satisfied. If so, we say that we have a validated solution, we apply Theorem~\ref{th:NK}: called $(\sigma,\bar{u},\bar{v})$ the validated approximate solution, we obtain a unique stationary solution $(\sigma,u,v)$ of~\eqref{eq:syst_para} that lives in a $r^*$-neighborhood. If not, increase $N$ and perform the proof on the part not yet validated.
The computational parts can also be reproduced by running the \textsc{Matlab} code \texttt{proof\_discrete\_diagram.m} from~\cite{GitHub}, together with \textsc{Intlab}~\cite{intlabRump}. 
Figure~\ref{fig:th-discrete-diag} can be obtained at the end of the script.

\section*{Declaration of interest}

The authors declare that they have no known competing financial interests or personal relationships that could have appeared to influence the work reported in this paper.

\section*{Acknowledgment}

MB and MP were supported by the ANR project CAPPS: ANR-23-CE40-0004-01.

\bibliographystyle{plain}
\bibliography{biblio.bib}
\end{document}